\documentclass{article}

\usepackage{color} %used for font color
\usepackage{amssymb} %maths
\usepackage{amsmath} %maths
\usepackage{epsfig}

\usepackage{amsthm}
\usepackage{enumitem}
\usepackage[dvipsnames]{xcolor}

\usepackage[backend=biber, style=numeric]{biblatex}
\usepackage{bbm}

\newtheorem{theorem}{Theorem}
\newtheorem{corollary}{Corollary}
\newtheorem{lemma}{Lemma}
\usepackage{hyperref}
\hypersetup{
    colorlinks=true,
    linkcolor=blue,
    filecolor=magenta,      
    urlcolor=cyan,
    pdftitle={Overleaf Example},
    pdfpagemode=FullScreen
    }

\usepackage{siunitx}

\usepackage{cleveref} 
\def\range{\mathcal{R}}

\def\zhatdef{{\mat{Z\Omega}_k^\dagger \mat{\Lambda}_{k}^{-1/2}}}
\def\zhat{\widehat{\mat Z}}

\def\reals{\mathbb{R}}
\def\R{\mathbb{R}}

\def\ld{\mathrm{logdet}}

\def\rank{\mathrm{rank}}
\def\trace{\mathrm{trace}}

\def\scriptO{{{\it O}\kern -.42em {\it `}\kern + .20em}}
\newcommand{\mat}[1]{\mathbf{#1}} 
\renewcommand{\vec}[1]{\mathbf{#1}}
\renewcommand{\t}{^\top}
\newcommand{\wh}[1]{\widehat{#1}} 
\newcommand{\e}[1]{\times 10^{#1}}

\newcommand{\bmat}[1]{\begin{bmatrix} #1 \end{bmatrix}}

\newcommand{\email}[1]{\protect\href{mailto:#1}{#1}}

\usepackage{algorithm}
\usepackage{algpseudocode}
\algrenewcommand\algorithmicrequire{\textbf{Input:}}
\algrenewcommand\algorithmicensure{\textbf{Output:}}
\usepackage{caption}
\usepackage{subcaption}
\usepackage[normalem]{ulem}

\Crefname{subsection}{Section}{Sections}
\crefname{subsection}{section}{sections}

\title{Optimal Experimental Design for Gaussian Processes via Column Subset Selection\thanks{This work was funded by the following grants: NSF DMS-1745654, NSF DMS-2309774, NSF DMS-2411198, and DOE DE-SC0023188.}}
\author{Jessie Chen\thanks{Department of Mathematics, North Carolina State University. Email: \email{jchen228@ncsu.edu}, \email{hji5@ncsu.edu}, \email{asaibab@ncsu.edu}} \and Hangjie Ji\footnotemark[2]\and Arvind K.\ Saibaba\footnotemark[2]}
\begin{document}
\maketitle

\begin{abstract}
    Gaussian process regression uses data measured at sensor locations to reconstruct a spatially dependent function with quantified uncertainty. However, if only a limited number of sensors can be deployed, it is important to determine how to optimally place the sensors to minimize a measure of the uncertainty in the reconstruction. We consider the Bayesian D-optimal criterion to determine the optimal sensor locations by choosing $k$ sensors from a candidate set of $n$ sensors. 
    Since this is an NP-hard problem, 
    our approach models sensor placement as a column subset selection problem (CSSP) on the covariance matrix, computed using the kernel function on the candidate sensor points. We propose an algorithm that uses the Golub-Klema-Stewart framework (GKS) to select sensors and provide an analysis of lower bounds on the D-optimality of these sensor placements. To reduce the computational cost in the GKS step, we propose and analyze algorithms for the D-optimal sensor placements using Nystr\"om approximations on the covariance matrix. Moreover, we propose several algorithms that select sensors via Nystr\"om approximation of the covariance matrix, utilizing the randomized Nystr\"om approximation, random pivoted Cholesky 
    and greedy pivoted Cholesky. We demonstrate the performance of our method on three applications: thin liquid film dynamics, sea surface temperature, and surrogate modeling.
\end{abstract}



\section{Introduction}
In many applications involving engineering and natural sciences, a key question is that of how best to collect data to reconstruct a spatially (or spatiotemporally) dependent function. Oftentimes, data can be collected using only a limited number of sensors due to budgetary or physical constraints. Thus, in this scenario, the question becomes how to best place the limited number of sensors to collect data and perform inference. This is known as optimal sensor placement and falls under the purview of optimal experimental design. 
In one formulation of optimal sensor placement, we are tasked with selecting $k$ best sensor placements out of $n$ candidate sensor locations (where $k\ll n$), so that we can make informative predictions of function values (as a scalar quantity) in locations without sensors. We assume that sensor locations must be selected before any data is collected, and that their locations are fixed throughout the data acquisition process. This is known as the batch or static setting~\cite{oedreview}.

A popular class of methods known as dictionary-based approach (see \cite{callaham2019robust,farazmand2023tensor} for a review of related techniques) can be used to estimate the state from sparse measurements. A prevalent example is POD-DEIM, which has a two-fold approach:  first, a basis $\mat\Phi$ is extracted from training data or high-resolution simulations, using a proper orthogonal decomposition (POD); and second,  an algorithm to place sensors and reconstruct the function is determined by the discrete empirical interpolation method (DEIM)~\cite{poddeim1,chaturantabut2010nonlinear}.    There are two drawbacks to these dictionary-based methods. First, to build such a dictionary, we would require large amounts of data over the entire spatial domain. Second, the reconstructions themselves are typically deterministic and do not come with quantifiable uncertainties in their predictions (some exceptions include~\cite{klishin2023data,kakasenko2025bridging}). We make clear comparisons between POD-DEIM and our proposed methods through numerical experiments in \Cref{sec:results}.

To address these two drawbacks, we turn to Gaussian Processes (GPs)~\cite{rasmussenGP,gramacy2020surrogates}, which have wide-ranging applications in spatial statistics, surrogate modeling, Bayesian inverse problems, etc. We introduce GPs more formally in \Cref{GPR}. GPs naturally come with quantifiable uncertainties and, compared to dictionary-based methods, require minimal training data for estimating hyperparameters. To determine optimal sensor placements, we adopt the so-called expected information gain (EIG) or D-optimality criterion. 

\begin{figure}
    \centering
    \includegraphics[width=0.75\linewidth]{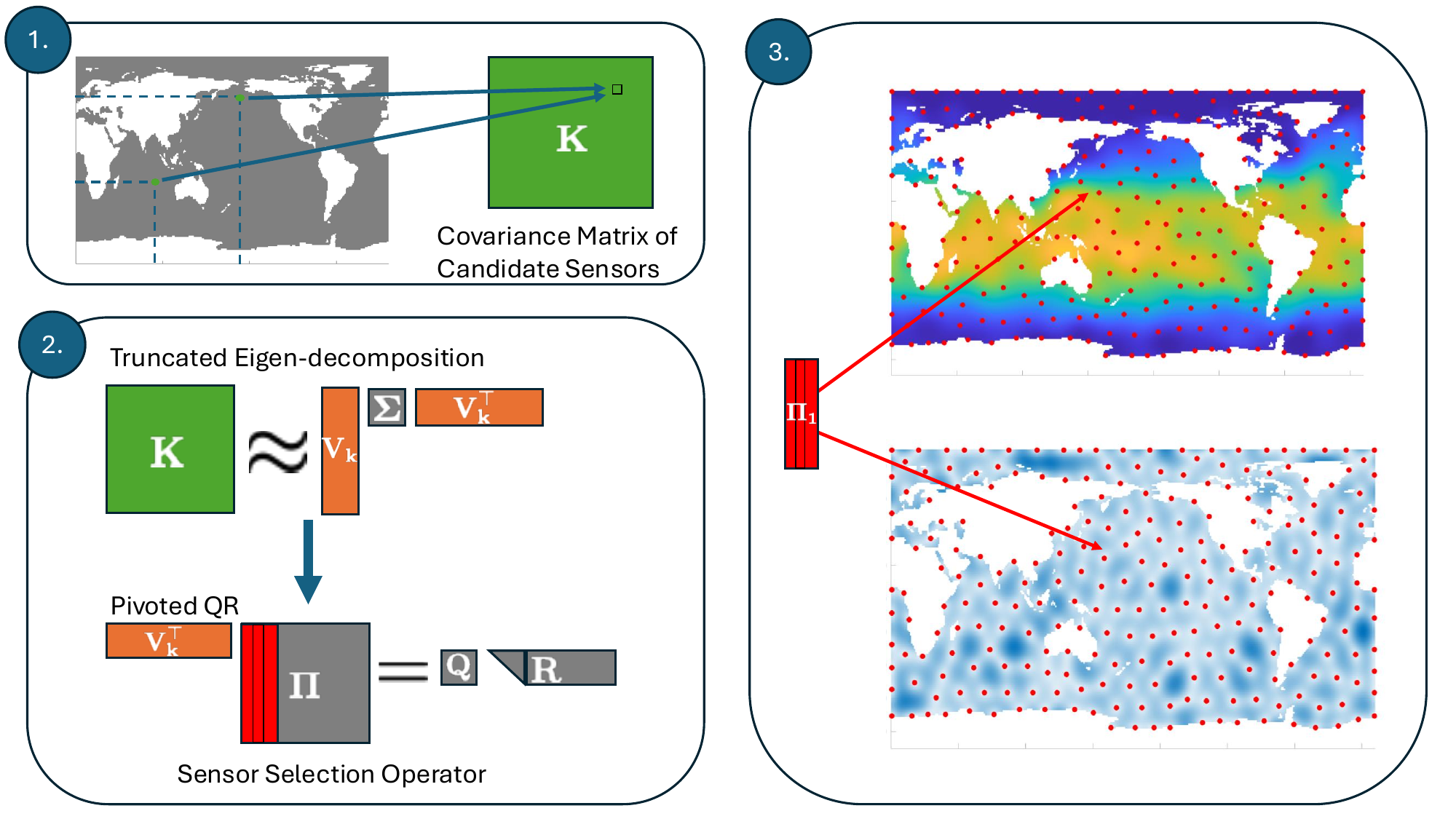}
    \caption{Schematic representation of the workflow in our proposed method to selecting D-optimal sensors in GPs. The red dots represent the sensors placed by the GKS algorithm on the covariance matrix $\mat K$.
    }
    \label{fig:lightning_slide}
\end{figure}

\subsection{Contributions and Outline of Paper}
In this paper, we model the data acquisition problem and reconstructions as a GP with quantifiable uncertainties. We follow the approach in~\cite{eswar2024bayesian} and view the sensor placement problem through the lens of column subset selection on the covariance matrix $\mat{K}$ defined using the kernel function and the candidate sensor locations. A schematic of this approach is given in \Cref{fig:lightning_slide}. The approach does not require observational data beyond calibrating the kernel function's hyperparameters, and the proposed algorithms are linear in computational complexity with respect to the number of candidate sensors $n$. We summarize the major contributions to this paper below. 
\begin{itemize}
\item \textbf{Conceptual GKS (\Cref{conceptual}):} We propose an algorithm that obtains near-optimal D-optimality via the Golub-Klema-Stewart (GKS) framework~\cite[Section 5.5.7]{golub2012matrix}, previously proposed for column subset selection. Furthermore, we provide theoretical bounds on the D-optimality criterion in terms of the partitioned eigendecomposition of the covariance matrix $\mat{K}$.  We say this approach is conceptual because for large $n$, this algorithm becomes infeasible.
\item \textbf{Nystr\"om Approximation and GKS (\Cref{ssec:nys})}: If the number of candidate sensors is large, the corresponding covariance matrix $\mat K$ is too large to be formed explicitly. Thus, we propose an algorithmic framework that first approximates the  covariance matrix using a randomized Nystr\"om approximation, followed by the conceptual GKS approach on the approximate covariance matrix.   We consider several flavors of Nystr\"om approximation including random projection, greedy pivoted Cholesky, and random pivoted Cholesky.  We provide bounds on the D-optimality of the selection made by the random projection and the greedy pivoted Cholesky algorithm. Potentially of separate interest, we derive bounds on the eigenvalues of a Nystr\"om approximation \Cref{sec:proofs}.

     \item \textbf{Numerical Experiments (\Cref{sec:results}):} We provide results from numerical experiments with contexts in one-dimensional fiber coating dynamics and two-dimensional sea surface temperature. A brief exploration on four-dimensional surrogate modeling is in \Cref{ssec:4d}. Each experiment utilizes the square exponential kernel. Our results show that using the GKS framework to place sensors is more optimal than random placements in terms of the D-optimality criterion, and comparable to the greedy approach. We also show that sensor placements using our algorithms give more accurate reconstructions in terms of relative reconstruction error over time compared to sensor placement using POD-DEIM \cite{poddeim1,poddeim2,qdeim}.
\end{itemize}

Lastly, \Cref{sec:proofs} compiles the proofs for the theorems in \Cref{sec:methods} and provides the reader with the necessary background concepts in matrix theory.\\
 The Supplementary Materials provide further numerical experiments on the one-dimensional fiber coating dynamics \Cref{ssec:gksvsnogks} and two-dimensional sea surface temperature \Cref{ssec:sst} examples. In addition, a brief exploration of a four-dimensional example for surrogate modeling is in \Cref{ssec:4d}. Finally, in the Supplementary Materials, a proof of the bounds on the eigenvalues of a Nystr\"om approximated matrix can be found in \Cref{sec:proofs}.

\subsection{Related work} In \cite[Section 2.3.3]{oedreview}, the authors summarize a variety of approaches to sensor placement and optimal experimental design (OED) in GPs. This list includes maximum entropy sampling \cite{maxentsample}, integrated variance in continuous interpretations of sensor placements in GPs \cite{gpexp}, and spectral decompositions of the Karhunen-Lo\'eve expansion of GPs \cite{spectralgpoed}.  Krause et al \cite{krause2008gpsensor} introduced a greedy selection algorithm for sensor selection in GPs in terms of mutual information between the selected and unselected sensors.

In the list above, the D-optimality criterion we consider is most closely related to the maximum entropy criterion. Ko et al.~\cite{ko1995exact} established that this problem is NP-hard and proposed greedy and exchange algorithms to solve this maximization problem. In a greedy algorithm, the optimal sensors are selected one at a time. In contrast, in an exchange approach, there are two sets of sensors: a chosen set and a discarded set. One or more sensors are exchanged to improve the criterion, till no further improvements are possible. Both the greedy and exchange algorithms are used as initial guesses in a branch-and-bound algorithm. The greedy algorithm \cite[Algorithm 1]{chen2018fast} we consider is mathematically equivalent, but more computationally efficient.   

More recent work has proposed variants of greedy algorithms with applications to Gaussian processes, such as distributed greedy and stochastic greedy~\cite{mirzasoleiman2013distributed, mirzasoleiman2015lazier}. Both papers focus on efficiently maximizing monotonic submodular functions. In \cite{mirzasoleiman2013distributed}, a greedy algorithm suited for parallel processing is presented. An efficient stochastic greedy algorithm is presented in \cite{mirzasoleiman2015lazier} with theoretical guarantees. These works apply their respective algorithms to GP sensor placement using the same criterion that we consider --- the D-optimality criterion, which maximizes the expected information gain. Our algorithms offer a different perspective on sensor placement. Namely, we view sensor placement based on selecting a near-optimal submatrix to preserve the spectrum of the covariance matrix. The proposed methods rely on low-rank approximations and pivoted QR factorizations, for which there are many highly optimized codes available. While the proposed methods have the same asymptotic cost as efficient implementations of the greedy method, in numerical experiments, our methods are slightly better than the greedy methods in terms of D-optimality.

GP regression can also be viewed as a linear inverse problem with a forward operator which is a selection or measurement operator (that is, it measures the point values of the function at prespecified locations), and with a GP prior. A review of techniques of sensor placement in Bayesian inverse problems can be found~\cite{oedreview,alexanderian2021optimal}. However, the existing methods for inverse problems  take into account the expense of the forward operator, and thus do not fully exploit the structure of the GPs.

In contrast to the static sensor placement which is the focus of this paper, in sequential sensor placement, a sequence of sensor placements are conducted, and the result of the previous sensor placement is used to inform the placement of the subsequent sensors. Related to sequential sensor placement is \emph{active learning} \cite{oedreview, gramacy2020surrogates}, which involves sequential (or batch) queries of data collection followed by a model update. For an active learning procedure with an acquisition function based on the D-optimality criterion, we posit that a selection made by our proposed algorithms provides selections in each acquisition loop. 
Seo et al \cite{seo} connected active learning perspectives to sensor placement for GPs and coined the terms \emph{Active Learning MacKay} \cite{ALMacKAy} and \emph{Active Learning Cohn} \cite{ALCohn}. Active Learning MacKay begins with no initial sensor placements and sequentially places sensors based on the next input data point with the largest predictive variance. This method approximates maximum entropy sampling. 
Active Learning Cohn is another sequential design that places sensors based on minimized average variance. Although our approaches are not directly relevant for the sequential or active learning setting, they can be used to seed an initial set of sensor placements at which to collect data.  

\section{Background}
This section briefly covers notation,  preliminary concepts on GPs and Rank-Revealing QR factorizations that are used throughout this paper.
\subsection{Notation}\label{notation}
In this paper, we denote matrices by a bold uppercase letter (e.g., $\mat{A}, \mat{B}$) and vectors by a bold lowercase letter (e.g., $\vec{x}, \vec{y}$). The $n \times n$ identity matrix is denoted by $\mat I_n$ whose $j$th column is represented by the basis vector $\vec e_j$, for $1\le j \le n$. The spectral norm of a matrix $\mat A$ is denoted by $\|\mat A\|_2$, whereas the 2-norm of the vector $\vec x$ is denoted by $\|\vec x\|_2$.

\paragraph{Permutation matrices and selection operator} \label{permutation} We denote a permutation matrix $\mat \Pi \in \reals^{n \times n}$ and create a partition $\mat \Pi = \begin{bmatrix}
    \mat S & \mat P
\end{bmatrix}$.  We define $\mat S = \begin{bmatrix}
    \vec e_{i_1} & \dots & \vec e_{i_k}
\end{bmatrix}\in \reals^{n \times k}$ with $k < n$ to be a \emph{selection operator} whose columns are extracted from the $n\times n$ identity matrix $\mat I_n$, corresponding to the distinct indices $\{i_1, \cdots, i_k\}$. The matrix $\mat P \in \reals^{n \times (n-k)}$ contains the ``unselected'' columns of the identity matrix. That is, $\mat P =  \begin{bmatrix}
    \vec e_{i_{k+1}} \dots \vec e_{i_n}
\end{bmatrix}$ whose columns also correspond to the $\{i_{k+1}, \cdots, i_n \}$ columns of the identity matrix $\mat I_n$.

\paragraph{Eigendecomposition} We denote the eigenvalues of a real symmetric matrix $\mat A \in \reals^{n \times n}$  as $\lambda_1(\mat A)$, $\cdots, \lambda_n(\mat A)$, which are placed in non-increasing order as $\lambda_1(\mat A) \geq \dots \geq \lambda_n(\mat A)$. If a matrix $\mat A \in \reals^{n \times n}$ is symmetric positive semi-definite, we write $\mat A \succeq \mat 0$. 

Now, let $\mat K = \mat V \mat \Lambda \mat V^\top$ be an eigendecomposition of the symmetric positive semi-definite matrix $\mat K \in \R^{n \times n}$, where $\mat V \in \reals^{n \times n}$ is orthogonal and $\mat \Lambda \in \reals^{n \times n}$ is a diagonal matrix whose entries are real-valued and non-negative eigenvalues of $\mat K$ in non-increasing order, $\lambda_1(\mat K) \geq \dots \geq \lambda_n(\mat K)$. Furthermore, we partition the decomposition as
\begin{align} \label{decomp}
    \mat K = \begin{bmatrix}
        \mat V_k & \mat V_\perp
    \end{bmatrix} 
    \begin{bmatrix}
        \mat \Lambda_k &  \mat 0 \\
        \mat 0 & \mat \Lambda_\perp
    \end{bmatrix}
    \begin{bmatrix}
        \mat V_k^\top \\ \mat V_\perp^\top
    \end{bmatrix} = \mat V_k \mat \Lambda_k \mat V_k^\top + \mat V_\perp \mat \Lambda_\perp \mat V_\perp^\top ,
\end{align}
where $\mat V_k \in \reals^{n \times k}, \mat V_\perp \in \reals^{n \times (n-k)}, \mat \Lambda_k \in \reals^{k \times k}, \mat \Lambda_\perp \in \reals^{(n-k) \times (n-k)}$.
We may also define the matrix square root as $\mat K^{1/2} = \mat V \mat \Lambda^{1/2} \mat V^\top$.

\paragraph{Singular Value Decomposition}  The singular value decomposition (SVD) of a matrix $\mat A \in \reals^{m \times n}$ is defined as $\mat A = \mat U \mat \Sigma \mat V\t$, where $\mat U \in \reals^{m \times m}$ and $\mat V \in \reals^{n \times n}$ have orthogonal columns, and $\mat \Sigma \in \reals ^{m \times n}$ is a diagonal matrix whose diagonal entries are the non-increasing nonnegative singular values $\sigma_1(\mat A) \geq \sigma_2(\mat A) \geq \dots \sigma_q(\mat A) \geq 0$ of $\mat A$, where $q = \rank(\mat A)$.
\subsection{Gaussian Processes}\label{GPR} 
Here, we give a brief introduction to GPs. GP regression, or Kriging, is a non-parametric and Bayesian regression technique and has applications including spatial statistics, machine learning, Bayesian optimization, and surrogate modeling \cite{oedreview,gramacy2020surrogates}.  

A GP is a collection of random variables (e.g., indexed by time or space), such that every finite collection of the random variables is a multivariate normal distribution. A GP is typically described by a mean $m(\vec x)$ and covariance function $\kappa (\vec x, \vec x')$ for any input vector $\vec x \in \R^D$, where $D$ is the dimension of the spatial domain. The resulting GP is denoted $\mathcal{GP}(m(\vec{x}), \kappa(\vec{x},\vec{x}'))$. Throughout this paper, we assume that the mean function $m(\vec{x}) = 0$. However, this is not a major limitation of the paper. The covariance or kernel function $\kappa(\vec x,\vec x')$ is user-defined and can embed properties or assumptions of the unknown function \cite{rasmussenGP}. Although our presentation will remain general in that any valid covariance matrix can be used, the numerical experiments will utilize the square exponential kernel, defined in~\eqref{gkernel}.

Suppose we are given a collection of points $\{\vec{x}_i\}_{i=1}^n$ that represent the candidate set of sensor locations. In the sensor placement task, we need to place $k$ sensors. Let the data be collected at indices $C = \{i_1,\dots,i_k\}$ and denote the remaining sensor locations as $T = \{i_{k+1},\dots,i_n\}$. 

At the sensor locations indexed by $C$, we have observational data corrupted by measurement noise, as 
\begin{equation} \label{eq:gp_setup}
    y_{i_j} = g(\vec{x}_{i_j}) + \epsilon_j \qquad 1 \le j \le k,
\end{equation}
where $\epsilon_j\sim \mathcal{N}(0,\eta^2)$ for $1 \le j \le k$ are independent realizations of the  additive measurement noise. The goal is to predict the function $g$ on the remaining points $\{\vec{x}_{i_j}\}_{j=k+1}^n$. To this end, we define the covariance matrix $\mat K \in \R^{n \times n}$ elementwise using the covariance function so that $[\mat K]_{ij} = \kappa(\vec x_i,\vec x_j)$  for $1 \le i,j \le n$, partitioned as 
\begin{align} \label{Kmatdef}
    \mat\Pi\t\mat K \mat\Pi =\begin{bmatrix}
        \mat{K}_{11}  & \mat{K}_{21}\t \\
        \mat{K}_{21} & \mat{K}_{22}
    \end{bmatrix} \quad \text{with sizes} \quad\begin{bmatrix}
        k \times k & k \times (n-k) \\
        (n-k) \times k & (n-k) \times (n-k)
    \end{bmatrix},
\end{align}
where $\mat\Pi$ is a permutation matrix
$$ \mat\Pi := \bmat{\vec{e}_{i_1} & \dots \vec{e}_{i_k} & \vec{e}_{i_{k+1}} & \dots \vec{e}_{i_n}} \in \reals^{n\times n} , $$
$\mat K_{11}$ is the covariance matrix involving the selected measurement locations $ \{\vec x_{i_j}\}_{j = 1}^k$ indexed by $C$. Similarly, the matrix $\mat K_{22}$ is the covariance matrix corresponding to the locations at which data is not collected, and the matrix $\mat K_{21}$ is the cross-covariance depending on the selected and unselected measurement locations.

Let $\vec{g}_p = \begin{bmatrix} g(\vec{x}_{i_{k+1}}) & \dots & g(\vec{x}_{i_n})\end{bmatrix}^\top \in \mathbb{R}^{n-k}$ denote the predicted function values at the locations denoted by $T$. Similarly, let $\vec{y} \in \reals^{k}$ denote the collected data and $$\vec{g}_c = \bmat{g(\vec{x}_{i_1}) & \dots & g(\vec{x}_{i_k})}\t  \in \reals^k$$ be the function values at locations indexed by $C$. For the regression task, the GP prediction for $\vec g_p| \vec y$ is described by the posterior distribution $\vec g_p|\vec y \sim \mathcal{N}\left(\vec{m}_p, \mat\Sigma_p \right)$, where the mean $\vec{m}_p$ and covariance $\mat\Sigma_p$ are defined as~\cite[Equations (2.23) and (2.24)]{rasmussenGP} 
\begin{align}\label{eq:meanrecon}
    \vec m_p  :=  \mat K_{21} [\mat K_{11} + \eta^2 \mat I]^{-1}\vec y, \qquad \mat\Sigma_p  := \mat K_{22} - \mat K_{21}[\mat K_{11} + \eta^2 \mat I]^{-1}\mat K_{21}\t.
\end{align}

\subsection{Expected Information Gain (EIG) and the D-Optimality Criterion} \label{ssec:dopt_intro}
We briefly discuss the notion of information gain in the context of data acquisition, which is related to sensor placement.

The criterion we consider for the sensor placement is the expected information gain (EIG). To define this criterion, let $\mat{K}_{11} \in \reals^{k\times k}$ be invertible.  Let $\nu_{\vec{f}_c}$ and $\nu_{\vec{f}_c|\vec{y}}$ denote the Gaussian measures corresponding to the prior $\vec{f}_c \sim \mathcal{N}(\vec{0}, \mat{K}_{11})$ obtained before any measurements are taken and the posterior distribution (obtained by conditioning on the data $\vec{y}$, takes the form $\vec{f}_c|\vec{y}\sim \mathcal{N}(\mat\Gamma\vec{y}, \mat\Gamma)$ ), where $\mat\Gamma = (\mat{K}_{11}^{-1} + \eta^{-2}\mat{I}_k)^{-1}$ is the posterior covariance matrix. 
The EIG criterion is defined as the expectation over the Kullback-Liebler divergence from the prior $\nu_{\vec{f}_c}$ to the posterior distribution $\nu_{\vec{f}_c|\vec{y}}$, with the expectation taken over the data $\vec{y}$. That is, the EIG takes the form~\cite[Section 1 and Theorem 1]{alexanderian2023briefnotebayesiandoptimality}
\[ \begin{aligned}  \overline{D_{\rm KL}} [\nu_{\vec{f}_c|\vec{y}}|| \nu_{\vec{f}_c}] := & \>  \mathbb{E}_{\vec{y}}\left\{ D_{\rm KL}  [\nu_{\vec{f}_c|\vec{y}}|| \nu_{\vec{f}_c}]\right\} \\ 
= & \> - \frac12\ld(\mat\Gamma) + \frac12\ld(\mat{K}_{11}) =  \frac12 \ld(\mat{I}_k + \eta^{-2}\mat{K}_{11} ). \end{aligned}\]

In the GP literature, this criterion is arrived at using a different line of reasoning. Given a random vector $\vec{z}\in \reals^n$ with density $\pi$, the differential entropy is defined as 
\[ H(\vec{z}) := -\int_{\reals^n} \pi(\vec{z})\log\pi(\vec{z}) d\vec{z}. \]
For a Gaussian random vector $\vec{z}  \sim \mathcal{N}(\boldsymbol\mu,\mat\Sigma)$, the differential entropy takes the form $H(\vec{z})= \frac{1}{2}\ld(2 \pi e \mat \Sigma)$. From~\eqref{eq:gp_setup}, the likelihood $\vec{y}|\vec{f}_c$ is a Gaussian distribution of the form  $\vec{y}|\vec{f}_c \sim \mathcal{N}(\vec{f}_c, \eta^2\mat I_k)$. Since $\vec{f}_c \sim \mathcal{N}(\vec{0}, \mat{K}_{11})$, the marginal distribution $\vec{y} \sim \mathcal{N}(\vec{0}, \mat{K}_{11} + \eta^2\mat I_k)$. One way of defining information gain is to consider the reduction in entropy during the data acquisition process measured as~\cite[Section 6]{gprtutorial}
\[ I(\vec{y}; \vec{f}_c) := H(\vec{y}) - H(\vec{y}|\vec{f}_c) = \frac12 \ld(\mat{I}_k + \eta^{-2}\mat{K}_{11} ). \] 
This is also the mutual information between the random variables $\vec{y}$ and $\vec{y}|\vec{f}_c$. The connection between the mutual information and the EIG was also noted in~\cite[Section 2.2.1]{oedreview}.

\subsection{Rank-Revealing QR Factorization}\label{ssec:qr}
Let $\mat{A}\in \reals^{m\times n}$ with $k \le \text{rank}(\mat{A})$, and consider a pivoted QR factorization of the form 
\begin{equation}
    \mat{A\Pi} =\mat{QR} = \begin{bmatrix}
        \mat Q_1 &\mat Q_2
    \end{bmatrix}\begin{bmatrix}
        \mat R_{11} & \mat R_{12} 
         \\ & \mat R_{22}
    \end{bmatrix}, 
\end{equation}
where $\mat\Pi$ is a permutation matrix, $\mat{Q}_1 \in \R^{m\times k}$ and $\mat{R}_{11}\in \reals^{k\times k}$. The size of all other matrices is clear from the context. Such a QR is called a rank-revealing QR factorization if 
\[ \begin{aligned}
  \frac{\sigma_k(\mat{A})}{p(n,k)} & \le \sigma_{\min}(\mat{R}_{11}) \le \sigma_k(\mat{A}) \\ 
  \sigma_{k+1}(\mat{A}) & \le \sigma_1(\mat{R}_{22}) \le p(n,k)\sigma_{k+1}(\mat{A}),
\end{aligned}
\]
where $p(n,k)$ is a function bounded by a low-order polynomial in $n$ and $k$. Note that the inequalities $\sigma_{\min}(\mat{R}_{11}) \le \sigma_k(\mat{A})$ and $\sigma_{k+1}(\mat{A})  \le \sigma_1(\mat{R}_{22})$ follow from interlacing singular value inequalities, and is true for any pivoted QR factorization. An algorithm, called the strong rank-revealing QR algorithm (sRRQR)~\cite{srrqr}, was proposed to obtain such a rank-revealing QR factorization for matrices with $m\ge n$. However, in this paper, we will need the version for short and wide matrices; see, for example, \cite[Section 3.1.2]{broadbent2010subset}. 

For a matrix $\mat A \in \reals^{k \times n}$ with $\rank(\mat A) =  k < n$ the QR decomposition with column pivoting takes the form 
\begin{align}\label{qrmats}
    \mat A \mat \Pi = \mat Q \mat R
    =  \begin{bmatrix}
        \mat Q_1 &\mat Q_2
    \end{bmatrix} \begin{bmatrix}
        \mat R_{11} & \mat R_{12} 
    \end{bmatrix}
\end{align}
where $\mat Q  = \begin{bmatrix}
    \mat{Q}_1 & \mat{Q}_2
\end{bmatrix}\in \reals^{k \times k}$ is orthogonal and $\mat R_{11} \in \reals^{k \times k}$ is upper triangular and nonsingular. The goal of a pivoted QR decomposition algorithm is to find a permutation matrix $\mat \Pi = \begin{bmatrix}
    \mat S & \mat P
\end{bmatrix} \in \reals^{n \times n}$ that partitions the columns of $\mat{A}$ to expose linear dependence. The sRRQR algorithm in Gu and Eisenstat \cite[Algorithm 4]{srrqr} can be modified appropriately for short and wide matrices to obtain the bounds   
\begin{align}\label{eqn:srrqr_bounds}
    \sigma_i(\mat A) \geq \sigma_i(\mat R_{11}) \geq \frac{\sigma_i(\mat A)}{p(n,k)}, \qquad 1 \le i \le k.
\end{align}
Here, the function $p(n,k) = \sqrt{1+f^2k(n-k)}$ depends on a user-specified parameter $f > 1$ that balances computational complexity and the performance of the algorithm. We note that the computational complexity of this algorithm for $f>1$ is $\mathcal{O}(k^2n\log_{f}n)$ flops. In practice, one may set $f =2$. While the sRRQR algorithm comes with the best known bounds, it is computationally expensive. An alternative approach is to use the QR with column pivoting (QRCP) algorithm from Businger and Golub \cite{cpqr_algorithm}.
This method greedily selects the next column among the trailing columns with the largest norm.  This algorithm has a computational complexity of $\mathcal{O}(nk^2)$ flops and is widely used in numerical experiments as it is the built-in algorithm for the \texttt{qr()} function in \texttt{MATLAB}. By using the algorithm, we obtain a factorization as in~\eqref{qrmats}, satisfying~\eqref{eqn:srrqr_bounds} but with $p(k,n) = \sqrt{n-k}\cdot 2^k$; see~\cite[Theorem 7.2]{srrqr}.

\section{Proposed Methods}\label{sec:methods}

 Before presenting the proposed methods, we refine our problem statement with the now established notation. Let $\{\vec x_i\}_{i=1}^n$ be the coordinate positions of the set of $n$ candidate sensors. 
Specifically, suppose that we have the ability to collect observations at only $k$ out of $n$ locations.
We want to ``optimally" select a subset of $k$ indices from the  $n$ available candidates. The criterion we consider is the EIG or D-optimality criterion as discussed in \Cref{ssec:dopt_intro}.

Thus, our task of choosing $k$ sensors is equivalent to the optimization problem\footnote{Note that in the definition of the D-optimality, we drop the factor of $\frac12$, since the optimizers are equivalent.}
\begin{align} \label{doptsub}
    \max_{\mat S \in \mathcal{P}_k}  \phi_D(\mat S) := \ld (\eta^{-2}\mat S\t\mat K\mat S + \mat I_k),
\end{align}
where $\mat{S} $ is a selection operator, $\mathcal{P}_k$ is the set of $\binom{n}{k}$ selection operators choosing $k$ columns of the identity matrix $\mat I_n$ and $\mat K $ as defined in 
\eqref{Kmatdef}. Note that $\mat{S}\t\mat{K}\mat{S} = \mat{K}_{11}$.  One can interpret \eqref{doptsub} as finding a submatrix of $\mat K$ such that $\phi_D(\mat S)$ is maximized. Our goal is to develop an algorithm to select a subset of $k<n$ input points which maximize $\phi_D$ over the set $\mathcal{P}_k$ with computational complexity that is linear in $n$. We model this problem then as a column subset selection (CSS) problem on the covariance matrix $\mat K$, similar to \cite{eswar2024bayesian}.

To summarize our approach, we use spatial data of the candidate sensor locations to form an $n \times n$ covariance matrix $\mat K$. We then apply CSS techniques to the covariance matrix $\mat K$ to place $k$ sensors. Lastly, observations from those selected sensor locations are used in a GP to predict the underlying function; see \Cref{fig:lightning_slide} for the schematic of our approach.
 We highlight the conceptual GKS algorithm (\Cref{conceptual}) along with three Nystr\"om approximated approaches to placing sensors in a D-optimal sense (\Cref{ssec:nys}). We also discuss an efficient greedy implementation in \Cref{ssec:greedy} as a benchmark for performance in \Cref{sec:results}.

\subsection{Conceptual GKS Algorithm}\label{conceptual}
We follow the GKS framework for column subset selection on the covariance matrix $\mat K$. This conceptual algorithm first performs an SVD on $\mat K$ to obtain the right singular vectors $\mat V \in \reals^{n \times n}$  of $\mat K$ (Note that since $\mat{K}$ is SPSD, the columns of $\mat{V}$ are also the eigenvectors of $\mat{K}$).  We denote by $\mat{V}_k$,  the  $k$ right singular vectors corresponding to the $k$ largest singular values. Then, a pivoted QR decomposition is applied to $\mat V_k\t$ as 
\[ \mat{V}_k\t \mat\Pi = \mat{V}_k\t \bmat{\mat{S}  & \mat{P} } = \mat{Q}_1\bmat{\mat{R}_{11} & \mat{R}_{12}}. \]
We partition $\mat\Pi = \bmat{\mat{S} & \mat{P} }$  and the matrix $\mat{S}$, corresponding to the first $k$ columns of $\mat\Pi$, determines the selection operator. \Cref{alg:cssp} outlines the conceptual GKS framework. Computing the SVD (or eigenvalue decomposition) requires $\mathcal{O}(n^3)$ flops. When using QRCP, the pivoted QR algorithm has a computational complexity of $\mathcal{O}(nk^2)$ flops.  The total computational complexity of this algorithm is, therefore, $\mathcal{O}(n^3+nk^2)$ flops. We call this a conceptual algorithm, since the initial step of computing the SVD is impractical for many large-scale applications.  

\begin{algorithm}[!ht]
\caption{Conceptual GKS Algorithm}\label{alg:cssp}
\begin{algorithmic}[1]
\Require Covariance matrix $\mat K$, number of sensors $k$
\Ensure sensor placement indices $\vec p$
\vspace{0.2cm}
\State $[\sim, \sim, \mat V] \gets \texttt{svd}(\mat K)$  \Comment{SVD of data matrix}  
\State $\mat V_k \gets  \mat V(:,1:k)$ \Comment{Store first $k$ columns of $\mat V$}
\State\label{lin:pivqr} $[\sim, \sim, \vec p] = \texttt{QR}(\mat V_k^\top, \texttt{`vector'})$ \Comment{Vector $\vec p$ determines indices of sensor placements from pivoted QR}
\State $\vec p \gets \vec p(1:k)$ \Comment{Store first $k$ columns of $\vec p$}

\end{algorithmic}

\end{algorithm}

We provide bounds on the objective function defined in \eqref{doptsub} that demonstrate the performance of the GKS approach. 
\begin{theorem}\label{thm1}
For any selection matrix $\mat S  \in \reals^{n \times k}$ such that $\rank(\mat V_k^\top \mat S)=k$,
\begin{align*}
     \ld(\mat I_k + \eta^{-2} \mat \Lambda_k)  \geq \phi_D(\mat{S}_{opt})  \geq \phi_D(\mat{S}) \geq \ld\left(\mat I_k + \frac{\eta^{-2} \mat \Lambda_k}{\|[\mat V_k^\top \mat S]^{-1}\|^2_2}\right),
\end{align*}
where $\mat\Lambda_k$ are the dominant eigenvalues as in~\eqref{decomp} and $\mat S_{opt} \in \reals^{n \times k}$ is the selection operator that selects columns of $\mat K$ which achieve maximal D-optimality.
\end{theorem}

\begin{proof}
    See \Cref{ssec:gksproof}.
\end{proof}

The first inequality gives an upper bound for the performance of the optimal selection operator $\mat S_{opt}$. The second inequality is obvious, and the last inequality quantifies the performance of the selection operator in terms of a lower bound involving the quantity $\|[\mat V_k^\top \mat S]^{-1}\|_2$. By assumption, $\mat V_k^\top \mat S$ is invertible and since $\mat{V}_k$ has orthonormal columns $\|[\mat V_k^\top \mat S]^{-1}\|_2\ge 1$. Thus, \Cref{thm1} suggests the heuristic that minimizing $\|[\mat V_k^\top \mat S]^{-1}\|_2$ can maximize $\phi_D(\mat{S})$. In fact, if we can find a selection operator $\mat{S}$ such that  $\|[\mat V_k^\top \mat S]^{-1}\|_2=1$, then $\phi_D(\mat{S}) = \phi_D(\mat{S}_{opt})$.

 In the second step of GKS, we may use sRRQR for the pivoted QR. The following corollary quantifies the performance of the algorithm. 
 
 \begin{corollary}\label{cor:dopt}
     Let $\mat{S}$ be the output of \Cref{alg:cssp} but with Line~\ref{lin:pivqr} replaced by sRRQR factorization~\cite[Algorithm 4]{srrqr} with parameter $f >  1$. Then, 
 \begin{align}\label{eqn:rrqr_bounds}
    \phi_D(\mat{S}_{opt}) \ge \phi_D(\mat{S})  \geq \ld\left(\mat I_k + \frac{\eta^{-2} \mat \Lambda_k}{1 + f^2k(n-k)}\right).
 \end{align}
 \end{corollary}
 \begin{proof}
     As in \Cref{lem:srrqrbounds}, we obtain the bound $\|[\mat{V}_k\t\mat{S}]^{-1}\|_2 \le  \sqrt{1 + f^2k(n-k)}$. The result follows by plugging this inequality into the result of \Cref{thm1}.
 \end{proof}

 However, the computational cost of the pivoted QR step in Line~\ref{lin:pivqr} of \Cref{alg:cssp} now becomes $\mathcal{O}(k^2n\log_f(n))$ flops. Thus, there is a tradeoff between a good lower bound and a high computational cost. Although QRCP greedily maximizes the volume of $\mat{V}_k\t\mat{S}$, we have found through numerical experiments that it does a good job of minimizing $\|[\mat V_k^\top \mat S]^{-1}\|_2$.

 As mentioned earlier, we consider this algorithm to be conceptual, since there is an underlying assumption that we can explicitly form and factorize the covariance matrix $\mat K$. In practice, if the number of sensors $n$ is sufficiently large, then the covariance matrix is too large to form. To this end, we consider Nystr\"om approximations of the covariance matrix. 

\subsection{Methods based on Nystr\"om Approximations}\label{ssec:nys}

The Nystr\"om approximation is a low-rank approximation of a positive semi-definite matrix. We follow the definition of the Nystr\"om approximation from \cite[Section 14]{rnla}.
Given a test matrix $\mat \Omega \in \reals^{n \times \ell}$, a Nystr\"om approximation of $\mat{K}$ is defined as
\begin{align}\label{eqn:nystrom}
    \widehat{\mat K} := (\mat K \mat \Omega) (\mat \Omega^\top \mat K \mat \Omega)^\dagger (\mat K \mat \Omega)^\top.
\end{align}
Note that the resulting approximation $\widehat{\mat K}$ is still symmetric positive semi-definite, i.e.,  $\widehat{\mat K} \succeq \mat{0}$~\cite[Lemma 2.1]{frangella2023randomized}. There are different flavors of the Nystr\"om approximations: 
\begin{enumerate}
    \item \textbf{Random projection}:  In this approach, we take the test matrix to be a Gaussian random matrix and to form the sketch $\mat{Y} = \mat{K \Omega}$. This approach is beneficial if the kernel matrix cannot be formed directly, but can only be accessed using matrix-vector products. This is the approach used in \Cref{ssec:nysgks}. 

    \item \textbf{Pivoted Cholesky}: In this approach, a column (or a block of columns) of the kernel matrix is selected, either greedily or randomly, to construct an approximation to the kernel matrix. As with random projection, this method shares the advantage that the entire matrix $\mat{K}$ need not be formed, but only formed as needed (i.e., a single column or a block of columns). We explore these approaches in \Cref{ssec:pivcholesky}.

    \item \textbf{Column sampling}: Another set of approaches involves having $\mat\Omega$ consist of columns of the identity matrix (perhaps with some additional scaling), which are sampled according to a certain probability distribution. Examples of these approaches include uniform sampling, leverage score sampling, ridge leverage score sampling, etc. The advantage of the sampling-based approaches is that the kernel matrix $\mat{K}$ need not be formed explicitly, but the columns can be generated and selected as needed. However, a downside of these approaches is that determining the sampling probabilities can be quite expensive (except for, perhaps, the uniform sampling distribution). We do not pursue this approach but it can be a viable alternative. 
\end{enumerate}

We provide bounds on the leading $k$ eigenvalues of a Nystr\"om approximated matrix in \Cref{thm2}. This may be of independent interest beyond this paper.  
\begin{theorem}
\label{thm2}

Let $\wh{\mat{K}}$ be the Nystr\"om approximation to $\mat{K}$ defined as in~\eqref{eqn:nystrom}. Further, define
\begin{align}
    \begin{bmatrix}
        \mat \Omega_k \\ \mat \Omega_{n-k}
    \end{bmatrix} : = \begin{bmatrix}
        \mat V_k^\top \mat \Omega \\ \mat V_\perp^\top \mat \Omega 
    \end{bmatrix}
\end{align}
with $\mat V_k, \mat V_\perp$ defined in~\eqref{decomp} and assume that $\rank(\mat \Omega_k)= k$.
Then, we have the following bounds on the eigenvalues of $\wh{\mat K}$ 
\begin{align} \label{nystrombounds}
    \lambda_i(\mat K) \geq \lambda_i(\wh{\mat K}) \geq \frac{\lambda_i(\mat K)}{1 + \lambda_i(\mat \Lambda_k^{-1})\|\mat \Lambda^{1/2}_\perp \mat \Omega_{n-k} \mat \Omega_k^\dagger\|_2^2} \qquad 1 \le i \le k.
\end{align}

\end{theorem}
\begin{proof}
    See \Cref{ssec:nysbound}.
\end{proof}

After obtaining a Nystr\"om approximation $\mat{K} \approx \wh{\mat{K}} = \mat{FF}\t$, we can determine the sensor placements by applying the GKS algorithm to $\wh{\mat{K}}$. However, since $\wh{\mat{K}}$ has a low-rank structure, the GKS step can be applied efficiently without forming $\wh{\mat{K}}$ directly.

\subsubsection{NysGKS: Nystr\"om combined with GKS}\label{ssec:nysgks}
In this approach, we combine the random projection approach for computing a Nystr\"om approximation with the GKS framework. To compute the Nystr\"om approximation, we draw a random test matrix $\mat \Omega \in \reals^{n \times (k+p)}$. Here $p \ge 0$ is a small oversampling parameter. We can construct the Nystr\"om approximation $\wh{\mat{K}}$ using~\eqref{eqn:nystrom}. The NysGKS Algorithm, then applies the GKS approach (\Cref{alg:cssp})  on the approximation $\wh{\mat{K}}$ to $\mat{K}$.

\begin{algorithm}[!ht]

\caption{NysGKS Algorithm}\label{alg:nys}
\begin{algorithmic}[1]
\Require  Kernel matrix $\mat K$,  number of sensors $k$,  and random test matrix $\mat \Omega \in \reals^{n \times (k+p)}$ where $p$ is the oversampling parameter
\Ensure Sensor placement indices $\vec p$
\vspace{0.2cm}
        \State $[\mat\Omega,\sim] = \texttt{QR($\mat \Omega, 0$)}$ \Comment{Thin QR Decomposition}
        \State $\mat Y = \mat K \mat \Omega$ 
        \State $\nu =\sqrt{n} \times 10^{-6}$ \Comment{Compute shift for stability}
        \State $\mat Y_\nu = \mat Y + \nu \mat \Omega$
        \State $\mat C = \texttt{chol($\mat \Omega\t\mat{Y_\nu}$)}$ \Comment{Compute Cholesky Decomposition}
        \State $\mat B = \mat{Y_\nu}/\mat C$ \Comment{Solve for a least-squares solution for $\mat B$ in $\mat B \mat C = \mat Y_\nu$}
        \State $[\wh{\mat U}, \sim, \sim] = \texttt{svd($\mat B$,0)}$ \Comment{Thin SVD}

\State $\wh{\mat U}_k \gets  \wh{\mat U}(:,1:k)$ \Comment{Store first $k$ columns of $\wh{\mat U}$}
\State $[\sim, \sim, \vec p] = \texttt{QR}(\wh{\mat U}_k^\top, \texttt{`vector'})$ \Comment{Get permutation matrix in vector form}
\State $\vec p \gets \vec p(1:k)$ \Comment{Store first $k$ columns of $\vec p$}

\end{algorithmic}
   
\end{algorithm}

However, there are two issues with this approach. First, a direct application of the formula is known to be numerically unstable. Second, forming $\wh{\mat{K}}$ is not computationally advisable. To remedy both issues,  we use the approach from~\cite[Algorithm 2.1]{frangella2023randomized} in Lines 1-7 of \Cref{alg:nys} to produce an eigendecomposition of $\wh{\mat{K}} = \wh{\mat{U}}\wh{\mat\Lambda}\wh{\mat{U}}^\top$ in factored format, with the eigenvalues of $\wh{\mat{K}}$ in decreasing order. Lines 3-4 of \Cref{alg:nys} maintain numerical stability by shifting the eigenvalues of $\mat Y$ away from 0 in machine precision. From here, we can obtain the basis $\wh{\mat{U}}_k$ by taking the first $k$ columns of $\wh{\mat{U}}$, and apply Lines 2-4 of \Cref{alg:cssp}. Details of this approach are given in \Cref{alg:nys}. In our numerical experiments, we use a Gaussian random test matrix\footnote{This means the entries are independent Gaussian random variables with zero mean and unit variance.} and an oversampling parameter $10 \leq p \leq 20$.

The dominant cost of \Cref{alg:nys} is in Line 2 in forming the sketch $\mat{K\Omega}$. To form this sketch na\"ively, the required computational cost is $\mathcal{O}(n^2k)$ flops.  The cost of the remaining steps in the algorithm are $\mathcal{O}(nk^2)$ flops, assuming that QRCP factorization is used in Line 9 of \Cref{alg:nys}. We may further reduce the computational cost of forming the sketch $\mat{K\Omega}$, in practice. If the points $\{\vec{x}_j\}_{j=1}^n$ are in  2 or 3 dimensions, and the kernels are sufficiently smooth, we can compute the matrix-vector products $\vec{x} \mapsto \vec{Kx}$ in $\mathcal{O}(n)$ or $\mathcal{O}(n\log n)$ flops using techniques such as $\mathcal{H}^2$-matrices and related techniques such as non-uniform FFT, fast multipole method, etc. In our numerical experiments, we use an implementation of $\mathcal{H}^2$-matrices in H2Pack~\cite{h2cite1,h2cite2}. The cost of the matrix-vector multiplication with $\mat{K}$ is now $\mathcal{O}(n)$ flops. Thus, the overall complexity of NysGKS with $\mathcal{H}^2$-matrices is $\mathcal{O}(nk^2)$ flops.  

In \Cref{thm3}, we provide probabilistic lower bounds of the D-optimality criterion using the NysGKS algorithm.  
\begin{theorem}\label{thm3}
 Let $\mat{S}$ be the output of \Cref{alg:nys} with a {Gaussian} test matrix $\mat \Omega \in \reals^{n \times (k+p)}$ with $p \ge 4$, and sRRQR with parameter $f > 1$ used in Line 9. Then, with probability at least $1- 3e^{-p}$,
\begin{align*}
     \phi_D(\mat{S}_{\rm opt}) \ge \phi_D({\mat{S}})  
    \geq \ld\left(\mat I_k + \frac{\eta^{-2}{\mat \Lambda}_k \mat{D}^{-1}}{(1+f^2k(n-k))}\right) \end{align*}
with $\mat{D} := \mat{I} +\mathrm{diag}\left(\frac{\gamma}{\lambda_k(\mat{K})},\dots, \frac{\gamma}{\lambda_1(\mat{K})}\right)$ and $$\gamma := \left[  \|\mat \Lambda_\perp\|_2^{1/2} \left( 16 \sqrt{1+ \frac{k}{p+1}}\right) + \sqrt{\trace(\mat \Lambda_\perp)} \frac{8\sqrt{k+p}}{(p+1)} \right]^2.$$
\end{theorem}
\begin{proof}
    See \Cref{ssec:nysproof}.
\end{proof}

The above theorem gives us bounds on the D-optimality of sensor placement using a randomized Nystr\"om approximation $\wh{\mat K}$ of the covariance matrix $\mat K$. Note that compared to \Cref{cor:dopt}, the lower bound is much more loose due to the presence of the additional term $\mat{D}$. This is due to the fact that the eigenvectors of $\mat{K}$ are now computed approximately. The use of sRRQR results in a higher computational cost $\mathcal{O}(nk^2 \log_f(n))$ flops, but with stronger theoretical guarantees than QRCP. In numerical experiments, however, we use QRCP, which has a computational cost of $\mathcal{O}(nk^2)$ flops. 

\subsubsection{Pivoted Cholesky} \label{ssec:pivcholesky}
An alternative approach to constructing a Nystr\"om approximation of $\mat{K}$ is using the pivoted Cholesky algorithm. This can then be combined with the GKS approach to determine the sensor placements.

To explain the pivoted Cholesky approximation, we define two sequences of matrices $\{\mat{K}^{(j)}\}$ and $\{\mat{E}^{(j)}\}$, which respectively contain the approximations and the residual matrices for $ 0 \le j \le k$. We initialize $\mat{K}^{(0)} = \mat{0}$ and $\mat{E}^{(0)} = \mat{K}$. At each step, we adaptively select an index $s_j \in \{1,\dots,n\}$ called a pivot, and construct the vector $\vec{w}_j = \mat{Ke}_{s_j}$ formed from the  column of $\mat{K}$ corresponding to the pivot $s_j$. We can use this column to update the approximation and the residual as 
\[\begin{aligned}
    \mat{K}^{(j)} = \> \mat{K}^{(j-1)} + \frac{\vec{w}_j \vec{w}_j^\top}{\vec{e}_{s_j}^\top\vec{w}_j}  \qquad \mat{E}^{(j)} = \> \mat{E}^{(j-1)} - \frac{\vec{w}_j \vec{w}_j^\top}{\vec{e}_{s_j}^\top\vec{w}_j} 
\end{aligned} \qquad 1 \le j \le k.\]
In practice, neither sequence of matrices $\{\mat{K}^{(j)}\}_{j=0}^k$ and $\{\mat{E}^{(j)}\}_{j=0}^k$ are formed explicitly, but they are stored in low-rank format.

There are several different strategies for selecting the pivots. We consider two such approaches in this paper.  In the \textbf{greedy pivoted Cholesky}~\cite{harbrecht2012low}, we pick an index from the largest entry in the residual matrix (which occurs on the diagonal, since $\mat{K}$ is SPSD), with ties broken in some fashion. That is, we take 
\[s_j \in \arg\min_{1 \le i \le n} \vec{e}_i^\top\mat{E}^{(j)}\vec{e}_i \qquad 0 \le j \le k. \]
This approach is also known as Cholesky with complete pivoting~\cite[Section 10.3]{higham2002accuracy}. The second approach we considered is \textbf{random pivoted Cholesky} (RPCholesky)~\cite{chen2025randomly}, which at step $j$ chooses the pivot $s_j$ based on the probability distribution 
\[ \mathbb{P}\{i = s_j\} = \frac{\vec{e}_i^\top\mat{E}^{(j)}\vec{e}_i}{\trace(\mat{E}^{(j)})} \qquad 1 \le i \le n. \]
Both methods have the advantage that they require forming only one column of the matrix $\mat{K}$ at a time. Thus, the run time of both algorithms is $\mathcal{O}(nk^2)$ flops.

\begin{algorithm}[!ht]
\caption{CholeskyGKS}\label{alg:cholgks}
\begin{algorithmic}[1]
\Require Covariance matrix $\mat K$ and number of sensors $k$
\Ensure Sensor placement indices $\vec p$ 
\vspace{0.2cm}
\State $\mat F = \texttt{pivotedCholesky}(\mat K, k)$ \Comment{$\mat F \in \reals^{n \times k}$ such that $\mat K \approx \mat F \mat F\t$}

\State $[\mat U, \sim, \sim] \gets \texttt{svd}(\mat F)$  \Comment{Thin SVD of data matrix} 
\State $\mat U_k \gets  \mat U(:,1:k)$ \Comment{Store first $k$ columns of $\mat U$}
\State $[\sim, \sim, \vec p] = \texttt{QR}(\mat U_k^\top, \texttt{`vector'})$ \Comment{Get permutation matrix in vector form}
\State $\vec p \gets \vec p(1:k)$ \Comment{Store first $k$ columns of $\vec p$}

\end{algorithmic}
\end{algorithm}

The output of both these approaches can be given in the form of an approximation, $\mat{K} \approx \wh{\mat{K}} = \mat{FF}^\top$ (See Step 2 in \Cref{ssec:choleskyeigs} for more details). We then apply the GKS framework in \Cref{alg:cssp} to $
\wh{\mat{K}}$. The step in Line 1 of \Cref{alg:cssp} can be simplified since the SVD of $\wh{\mat{K}}$ can be readily obtained from a thin SVD of $\mat{F}$. These steps are summarized in \Cref{alg:cholgks}. The first step involving the pivoted Cholesky can be computed using the greedy pivoting or RPCholesky. It is interesting to note that the pivots $\{s_j\}_{j=0}^k$ can also be considered as sensor placement locations, thus suggesting that the additional GKS step may not be necessary to determine sensor placements. However, in numerical experiments (\Cref{ssec:gksvsnogks}), we found that the additional GKS step to be beneficial in the sense that it yielded a higher D-optimality value.     

In what follows, we will provide an analysis of ChokeskyGKS (\Cref{alg:cholgks}), assuming that greedy selection was used to determine the pivots.  

\begin{theorem}\label{thm:greedypiv}
    Let $\mat{S}$ be a selection operator obtained from \Cref{alg:cholgks} with greedy pivoted Cholesky in Line 1, and sRRQR in Line 4. Then
    \[ \phi_D(\mat{S}_{\rm opt}) \ge \phi_D(\mat{S}) \geq \ld\left(\mat I_k + \frac{\eta^{-2}\mat \Lambda_k \mat C}{1 + f^2k(n-k)}\right),\] 
    where $\mat C \in \reals^{k \times k}$ is a diagonal matrix such that $c_{ii} = [4^{i-1}(n - i+1)]^{-1}$ for $1 \le i \le k$.
\end{theorem}
\begin{proof}
    See \Cref{ssec:choleskyeigs}.
\end{proof}
The interpretation of this result is similar to that of \eqref{eqn:rrqr_bounds}; however, the lower bound is much more pessimistic due to the additional scaling in the matrix $\mat{C}$. We were not able to derive an analogous result for RPCholesky, since existing results only apply to low-rank approximations and not to individual eigenvalues, as is required in the analysis of the GKS bounds. 

\subsection{Greedy Approaches}\label{ssec:greedy}
We discuss a greedy method for solving the maximization problem~\eqref{doptsub} and use this algorithm has a benchmark for our proposed algorithms in both theory and numerical experiments.
The method greedily selects a sensor that maximizes the D-optimality objective function at each iteration $j$ and terminates when the desired number of sensors $k$ has been reached. In the greedy algorithm, at least $\sum_{j=0}^{k-1} (n-j)$ computations of the D-optimality criterion must be made. The computation of the log determinant for each new sensor requires $\mathcal{O}(j^3)$ flops for $1 \le j \le k$. So, the total computational cost of this greedy algorithm is on the order of $\sum_{j=0}^{k-1}\mathcal{O} (j^3)(n-j) = \mathcal{O}(nk^4)$ flops. A more efficient implementation of this algorithm, which sequentially computes the determinant, is discussed next.

Chen et. al (\cite[Algorithm 1]{chen2018fast}) presents an efficient implementation that builds a submatrix of an SPD matrix that maximizes the log determinant. In that paper, the authors propose an algorithm to accelerate the greedy MAP inference for Determinantal Point Processes. While the context and application of the algorithm is different from the problem at hand, the action of the algorithm is the same. To the best of our knowledge, this algorithm as not been used for sensor placement in GPs. 
   For our problem, the input to \cite[Algorithm 1]{chen2018fast} is the SPD matrix $\mat A:= \mat I_n + \eta^{-2}\mat K$. The computational complexity is reduced from $\mathcal{O}(nk^4)$ flops to $\mathcal{O}(nk^2)$ flops by updating the Cholesky factor of the covariance matrix of the selected sensors in each iteration. We will refer to \cite[Algorithm 1]{chen2018fast} as the efficient greedy algorithm for the remainder of this paper.

We now discuss theoretical guarantees of using the greedy algorithm. The objective function $\phi_D(\mat S) $ is submodular  and monotonic, if viewed as a set function on the selected sensors and it can be shown that~\cite{nemhauser1978analysis,krause2014submodular,alexanderian2025brief} 
    \begin{align} \label{greedybounds}
        \phi_D(\mat S_{opt}) \geq (1-1/e) \max_{\mat S \in \mathcal{P}_k} \phi_D(\mat S).
    \end{align}

This result allow us to compare the theoretical guarantees of the greedy algorithm to our proposed algorithms. In the following section, we present numerical results of our proposed algorithms compared against the efficient greedy algorithm as a benchmark.

\subsection{Summary and comparison of computational costs}

\Cref{tab:compcost} summarizes the proposed methods by listing the algorithm reference in the paper and the corresponding computational cost. In the cost estimates, we assume that the pivoted QR has been done using QRCP. Except for the Conceptual GKS and NysGKS-na\"ive algorithms, the proposed algorithms have linear complexity in the number of candidate locations.  We note that the computational complexity of this efficient greedy algorithm is on the same order of magnitude as some of our proposed methods (see \Cref{tab:compcost}). However, our proposed algorithms are competitive both in runtime and numerical performance exemplified in \Cref{sec:results}. 
\begin{table}[!ht]
    \centering
    \begin{tabular}{c|c|c}
         Algorithm & Reference & Comp. Cost (flops)\\
         \hline
         Conceptual GKS & \Cref{alg:cssp}&  $\mathcal{O}(n^3+nk^2)$\\
         NysGKS-na\"ive & \Cref{alg:nys}&  $\mathcal{O}(n^2k+nk^2)$\\ \hline
         NysGKS + $\mathcal{H}^2$-matrices & \Cref{alg:nys} and \cite{h2cite1,h2cite2} &  $\mathcal{O}(nk^2)$\\
         RPCholesky + GKS & \Cref{alg:cholgks} and \cite[Algorithm 1]{chen2025randomly} &  $\mathcal{O}(nk^2)$\\
         Pivoted Cholesky + GKS & \Cref{alg:cholgks} &  $\mathcal{O}(nk^2)$\\ \hline 
         Efficient Greedy  & \cite[Algorithm 1]{chen2018fast} &  $\mathcal{O}(nk^2)$\\
         
    \end{tabular}
    \caption{Comparison of all proposed algorithms and their respective computational complexity in terms of floating point operations. Here, $n$ is the number of candidate sensors and $k$ is the number of available sensors.
    \label{tab:compcost}}
\end{table}

\section{Numerical Experiments} \label{sec:results}
In this section, we consider three applications: thin liquid films and sea surface temperature. An example on surrogate modeling can be found in \Cref{ssec:4d}. In the following experiments, the covariance function chosen for the GP is the squared exponential kernel defined by
\begin{align} \label{gkernel}
     \kappa(\vec x, \vec y) := \sigma_f^2\exp\left(-\frac{\|\vec x - \vec y\|_2^2}{2\ell^2}\right)
\end{align}
where $\sigma_f$ represents the signal variance hyperparameter, and $\ell$ represents the lengthscale hyperparameter.

Furthermore, we evaluate the performance of our algorithms using the D-optimality criterion or the objective function in \eqref{doptsub}, and the relative reconstruction error in Euclidean norm, computed using the formula
\begin{align}\label{relerr}
    \text{error} = \frac{\|\vec m(\vec f_p)-\vec f\|_2}{\|\vec f\|_2},
\end{align}
where $\vec m(\vec f_p)$ is defined in \eqref{eq:meanrecon} and $\vec{f}$ is the vector containing the true function values.

\paragraph{Computational environment} All computations  are carried out in \texttt{MATLAB2024a} on MacOS Sequoia v15.5 with an Apple M1 chip and 32GB of Memory. The exception is a numerical experiment (see \Cref{costvsn} which compares the computation time of NysGKS + $\mathcal{H}^2$-matrices (\Cref{alg:nys} and \cite{h2cite1,h2cite2}) as the number of candidate sensors increases, which is computed on an HPC cluster.   

\paragraph{Hyperparameter tuning}
To find the length scale $\ell$ and signal variance $\sigma_f$ hyperparameters for the covariance functions in \eqref{gkernel}, we perform a sweep of potential values. Given a range of values for the lengthscale $\ell$ and signal variance $\sigma_f$, we form a 2D grid. For each element in the grid, we compute the log marginal likelihood 
\begin{align}\label{lml}
    \log p(\vec y|\mat X, \sigma_f, \ell) = -\frac{1}{2}\vec y\t (\mat K + \eta^2 \mat I_n)^{-1}\vec y - \frac{1}{2} \log \det  (\mat K + \eta^2 \mat I_n) - \frac{n}{2} \log 2\pi
\end{align}
using the full covariance matrix (or in the case of the second and third numerical experiment, an approximation of the covariance matrix) and the provided observation data $\vec y$ which is specified in each numerical experiment. The hyperparameter pair which achieves the maximum log likelihood is then used to form the kernel function for numerical experiments. In the second and third applications, the number of candidate sensors $n$ is so large such that the covariance matrix can not be explicitly formed. Thus, we use a rank-$k$ Nystr\"om approximation $\wh{\mat K}$ on the covariance matrix $K$ using the RPCholesky algorithm \cite{chen2025randomly} before computing the log marginal likelihood for each candidate hyperparameter pair. The computation of the log-likelihood is further accelerated by using the Woodbury identity and the matrix determinant lemma \cite[Appendix A.3]{rasmussenGP}.

\subsection{Thin Liquid Film Dynamics - A 1D Case Study}
The first application is a one-dimensional example of thin liquid film dynamics. The data come from a numerical solution of a dimensionless fourth-order nonlinear PDE model, presented in \cite[(3.14)]{dropletdata}, which describes the dynamics of the free surface height $h(x,t)$ of a thin liquid film flowing down a vertical cylindrical fiber. This model incorporates gravity, the streamwise and azimuthal curvature of the interface, and the substrate geometry. For this example, we set the system parameters to $\alpha = 2.299$, $\eta = 0.231$, and $A = 0$. The domain $x \in [0,  10]$ is discretized with $n = 6001$ equally spaced grid points, which we also use as candidate sensor locations. The PDE is numerically solved using a fully-implicit second-order finite difference method. 
The dynamics of the liquid film are captured with $31$ time snapshots at times $t=0,10,\dots,300$. Starting from a nearly flat initial profile with small perturbations, the numerical solution exhibits early-stage transient behavior, followed by the formation of an array of three droplets that begin to travel down the fiber at approximately $t=140$. A typical solution profile from the PDE simulation at $t = 140$ is shown in \Cref{fig:recon1d}.  We used this snapshot in the sweeping procedure for the hyperparameters.  The noise variance $\eta^2$ was chosen based on $0.2\%$ noise. The following table lists the parameter values of $\sigma_f$, $\ell$, and $\eta^2$ used in all the experiments

\begin{center}

    \begin{tabular}{c|c|c}
 
        $\sigma_f$ & $\ell$ & $\eta$ \\ \hline
        1  & $0.5$ &  $4.2784\e{-4}$\\
 
    \end{tabular}

\end{center}
Additional numerical experiments can be found in \Cref{sec:results}. These experiments include the statistics of the relative reconstruction error \eqref{relerr} over time and the effect of the number of selected sensors $k$ on the D-optimality score and the relative reconstruction error \eqref{relerr}.

\subsubsection{Performance of Conceptual GKS}
First, we demonstrate the Conceptual GKS \Cref{alg:cssp} on the liquid film example using the time snapshot $t = 140$ (see \Cref{fig:recon1d}), with $k=30$ sensors. 
The  mean reconstruction of the GP is shown as a red dashed curve in the left panel of \Cref{fig:recon1d}. It closely matches the true liquid film profile (solid blue curve) across the entire domain. The gray-shaded region represents the point-wise one-standard deviation interval in the domain.  Also displayed are the selected sensors in red dots. The D-optimality of this sensor selection is $221.39$ and the relative reconstruction error is $0.0025$.

\begin{figure}[!ht]
    \centering
    \includegraphics[width=0.45\linewidth]{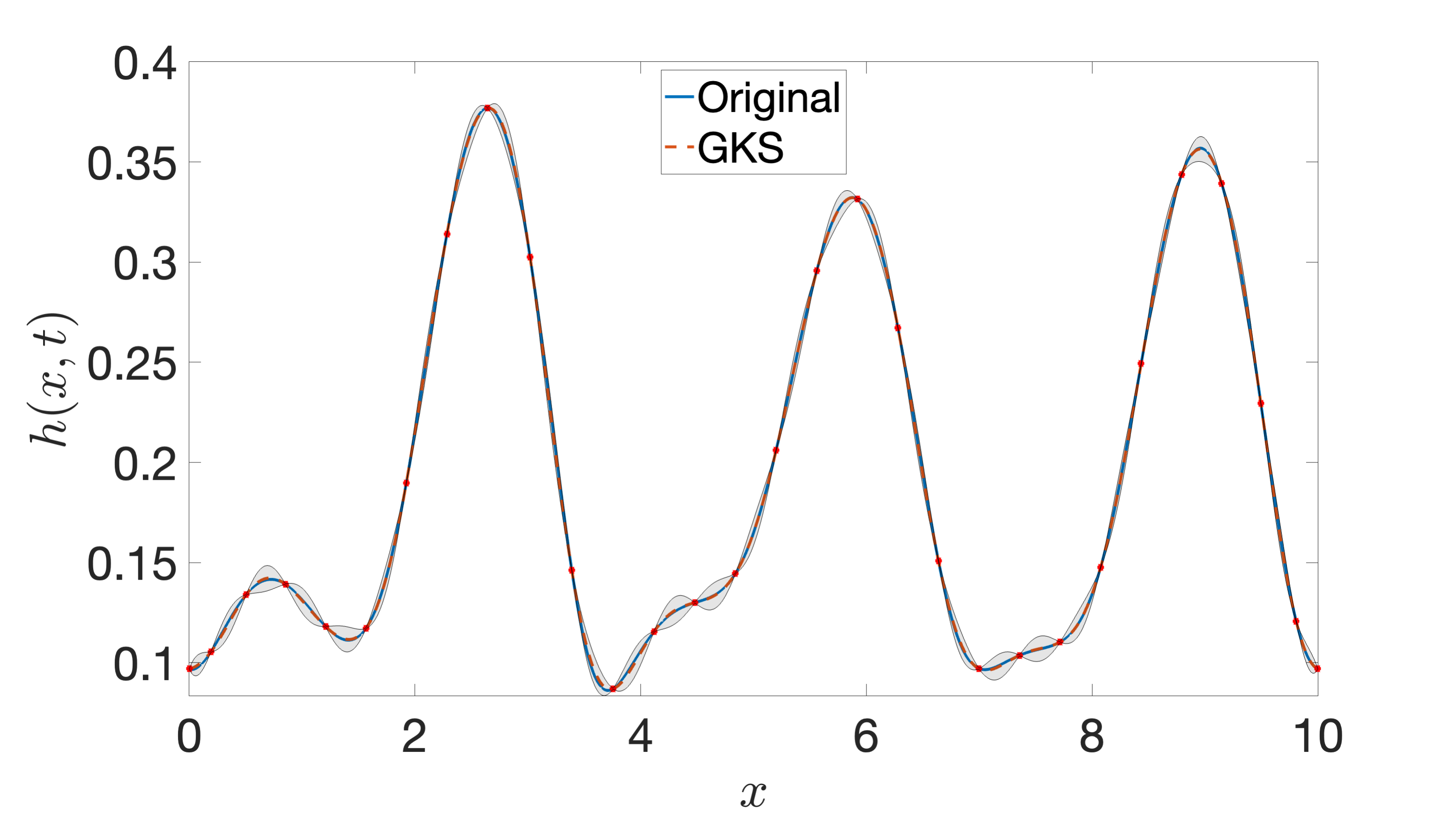}
    \includegraphics[width=0.45\linewidth]{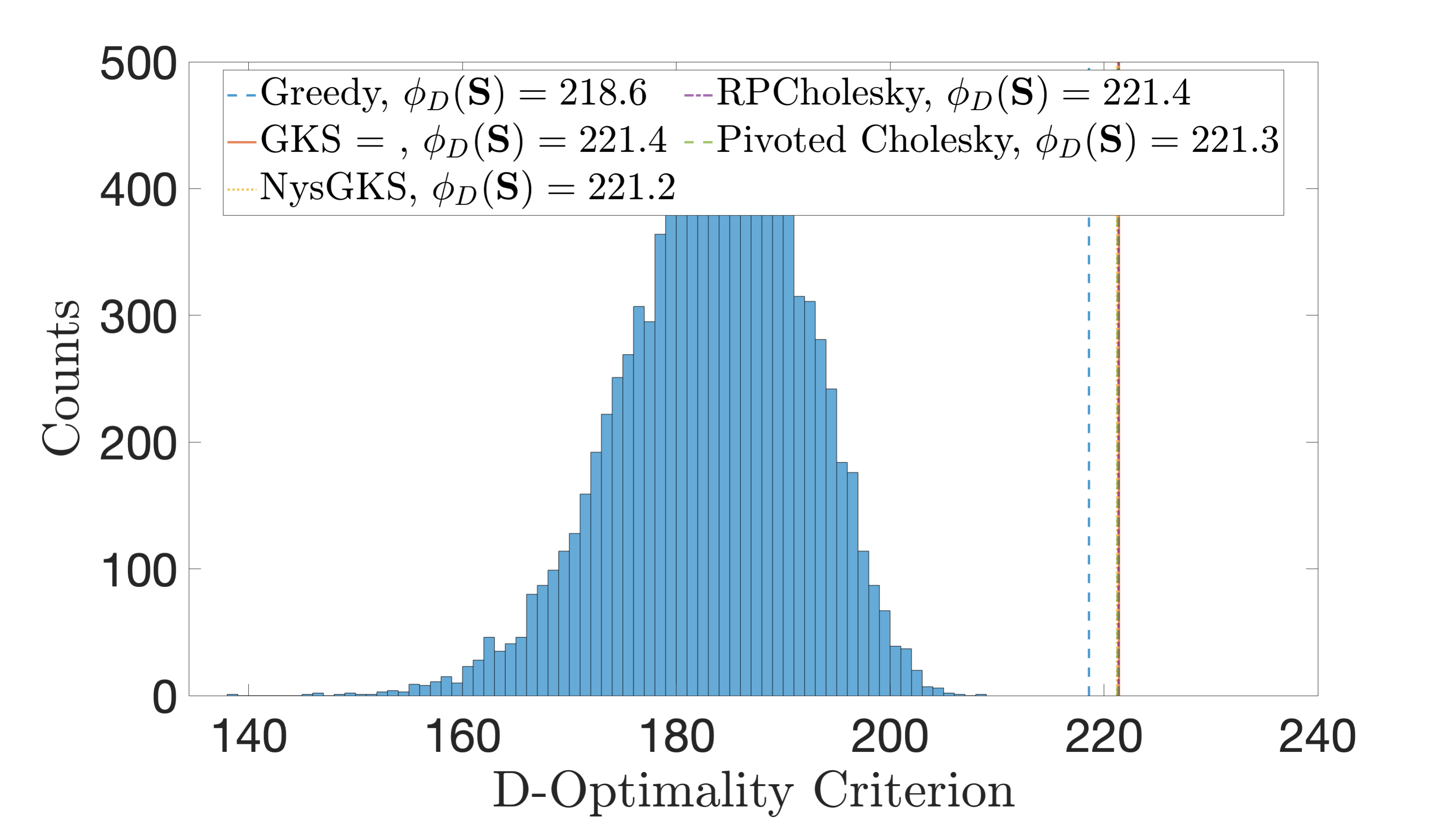}
     
    \caption{(Left) GP regression in the dashed red line and sensor selection in red asterisks using the conceptual GKS (\Cref{alg:cssp}). Reconstruction of the function is at $t=140$. The sensor selection achieves a D-optimality score of 221.39. The relative error \eqref{relerr} in the reconstructed droplet profile is 0.0025. (Right) Conceptual GKS (\Cref{alg:cssp}), efficient greedy \cite[Algorithm 1]{chen2018fast}, rank $k+10$ NysGKS (\Cref{alg:nys}), and rank $k$ RPCholesky and greedy pivoted Cholesky (\Cref{alg:cholgks}) compared with 10,000 realizations of random sensor placement in terms of D-optimality.}
    \label{fig:recon1d}
\end{figure}

Next,  we compare the performance of the conceptual GKS (\Cref{alg:cssp}), Greedy \cite[Algorithm 1]{chen2018fast}, rank $k$ NysGKS Algorithm (\Cref{alg:nys}) with oversampling parameter $p=10$, and rank $k$ RPCholesky and greedy pivoted Cholesky (\Cref{alg:cholgks}) in terms of the D-optimality score from each respective sensor placement. These scores are compared against the D-optimality scores achieved by 10,000 realizations of random sensor placements. For each method, we fix the number of available sensors as $k = 30$. 

The histogram of D-optimality scores from the randomly placed sensors is shown in \Cref{fig:recon1d} (right). We observe from the solid red line that the conceptual GKS algorithm achieves the highest score in terms of D-optimality of 221.39. This is followed by the rank $k$ RPCholesky at 221.36, rank $k$ greedy pivoted Cholesky at 221.28, rank $k+10$ NysGKS at 221.24, and lastly, the greedy algorithm at 218.58. Overall, we see that each algorithm outperforms the random sensor selection in terms of D-optimality.

\subsubsection{Error over time}
Next, we analyze the relative reconstruction errors throughout the time interval $t = 0, \dots, 300$. 

\begin{figure}[!ht]
    \centering
    \includegraphics[width=0.9\linewidth]{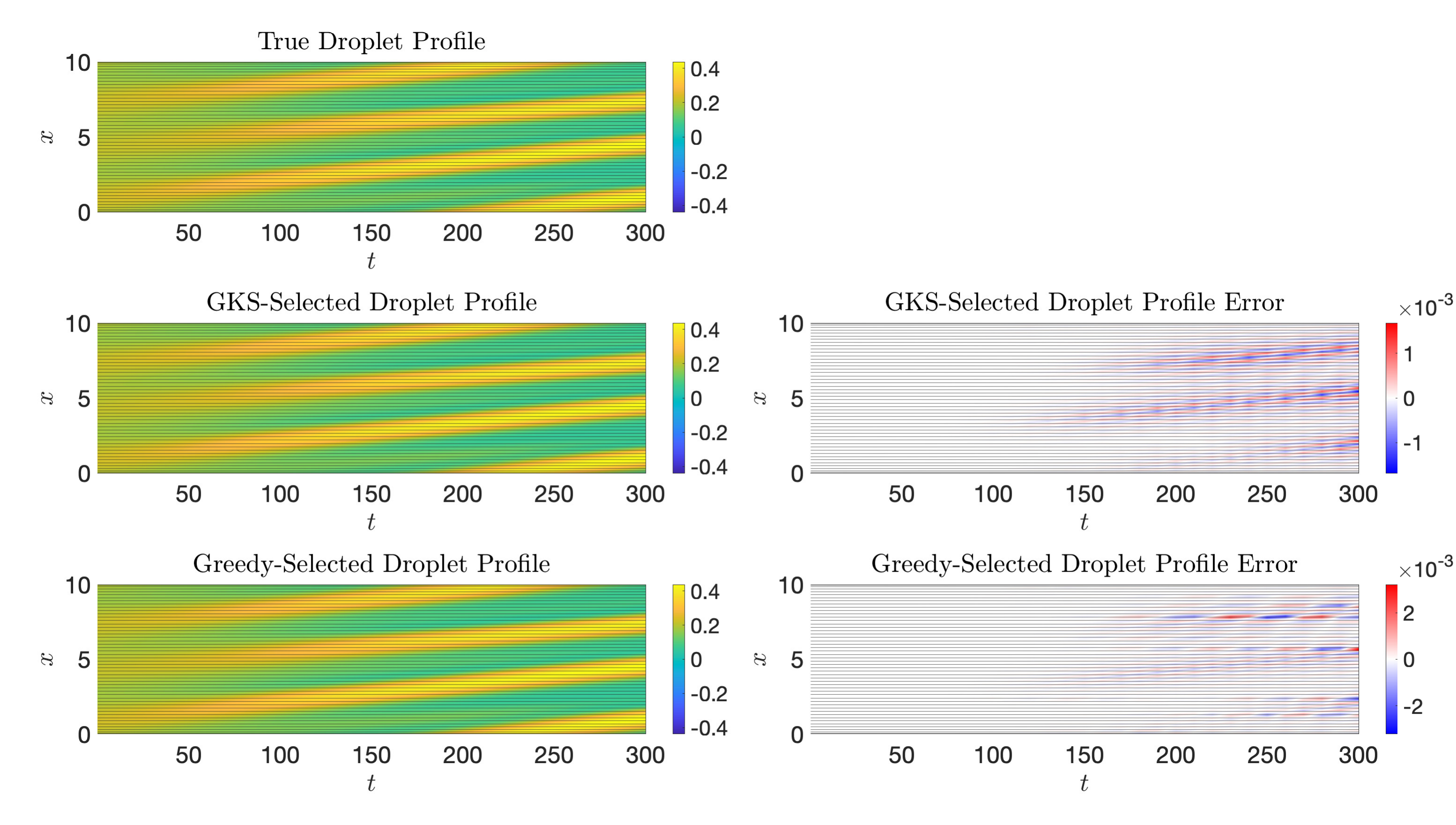}
    \caption{Comparisons of the height of the droplet over the entire domain in time and space for the true profile, GKS-selected sensor reconstruction (first column) and its error (second column), and Greedy-selected sensor reconstruction and its error. The black horizontal lines represent the sensor selections for each respective algorithm.}
    \label{fig:contourall}
\end{figure}

To further compare the conceptual GKS algorithm (\Cref{alg:cssp}) and the Greedy algorithm \cite[Algorithm 1]{chen2018fast}, we plot each reconstruction over time in \Cref{fig:contourall} along with their respective sensor placements. Since this numerical experiment is 1D, we plot the height of the droplet on a color scale along the spatial domain (represented on the vertical axis) and assign the horizontal axis to the temporal domain. We also provide the corresponding error plots on the side. The black horizontal lines represent the placed sensors to collect observations. As the time $t$ increases, the dynamics of the droplet evolve and become more difficult for the chosen kernel and hyperparameters to capture.

\subsection{Sea Surface Temperature - A 2D Case Study}
The second experiment is a two-dimensional example of reconstructing sea surface temperature from sensor observations. We use the NOAA Optimum Interpolation (OI) SST V2 data provided by the NOAA PSL, Boulder, Colorado, USA, from their website at \url{https://psl.noaa.gov}. We use sea surface temperature data collected weekly by the NOAA from the beginning of 1990 until the week of January 29, 2023 \cite{sst}. 
The two-dimensional domain covers all longitudinal and latitudinal degrees, $[0.5,359.5] \times [-89.5, 89.5]$ in increments of 1 degree, yielding a total of $64,000$ grid points. This includes land areas which are masked in the data. The number of points in the domain associated with the sea is $44,219$, which serve as our candidate sensor locations.

\begin{figure}[!ht]
    \centering
    \includegraphics[width=0.75\linewidth]{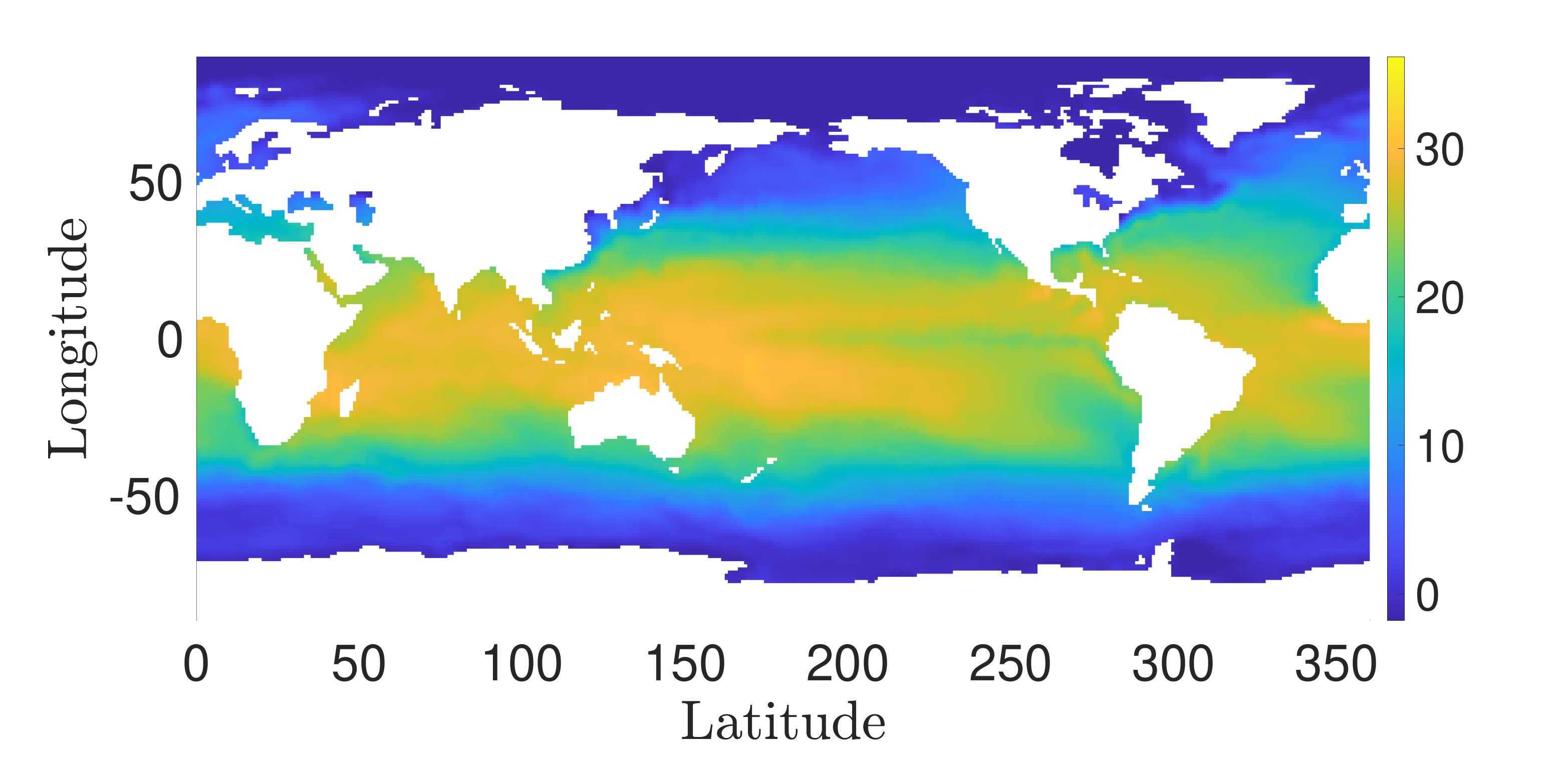}
    \caption{True Sea Surface Temperature for the week of February 4, 2018.}
    \label{fig:t5true}
\end{figure}

For this problem, the number of candidate sensors is too large for the conceptual Algorithm to be feasible. Thus, we report results using the NysGKS (\Cref{alg:nys}) and compare against the results from the efficient greedy algorithm (\cite[Algorithm 1]{chen2018fast}). To further speed up computation time, we use the implementation in H2Pack \cite{h2cite1, h2cite2} which allows for matrix-vector multiplication of dense kernel matrices in linear time. 
The true sea surface temperature data at this time snapshot is shown in \Cref{fig:t5true}. 
The following hyperparameters were used in the following experiments: 
\begin{center}
    
    \begin{tabular}{c|c|c}

        $\sigma_f$  & $\ell$& $\eta$ \\ \hline
        $0.11$& $16$&  $0.033$
    \end{tabular}

\end{center}
The parameters $\sigma_f,\ell$ were obtained by the parameter sweep described at the start of the section on sea surface temperature data in the week of February 4, 2018, whereas the noise variance $\eta^2$ is based on $0.2\%$ additive Gaussian noise. Additional numerical experiments can be found in \Cref{ssec:sst}. This encompasses the effect on the the relative reconstruction error \eqref{relerr} and D-optimality score as $k$ increases, the computation time as the number of candidate sensors increases, and a further comparison to he POD-DEIM algorithm \cite{poddeim1,poddeim2,qdeim}. Our proposed algorithms are compared to the POD-DEIM algorithm through the relative reconstruction error \eqref{relerr} over time by showcasing both normalized and un-normalized results. Lastly in \Cref{ssec:sst}, we explore the selection results when the kernel function \eqref{gkernel} is modified to use the great circle distance rather than Euclidean distance to account for the shape of the Earth.

\subsubsection{Performance of Methods}
The near-optimal sensor placements from \Cref{alg:nys} for $k=250$ is displayed using red asterisks in \Cref{fig:2drecon} (left). In the same figure, we plot the GP mean prediction on the week of February 4, 2018. In \Cref{fig:2drecon} (right) we plot the corresponding one unit of standard deviation. As expected, the uncertainty near the placed sensors is lower than the uncertainty where no sensors are placed nearby, such as the Mediterranean Sea. The relative error in the reconstruction \eqref{relerr} is $ 0.0781$ and the D-optimality is $1074.8$.
\begin{figure}[!ht]
     \centering
     \begin{subfigure}[b]{0.49\textwidth}
     
         \centering
         \includegraphics[width=\textwidth]{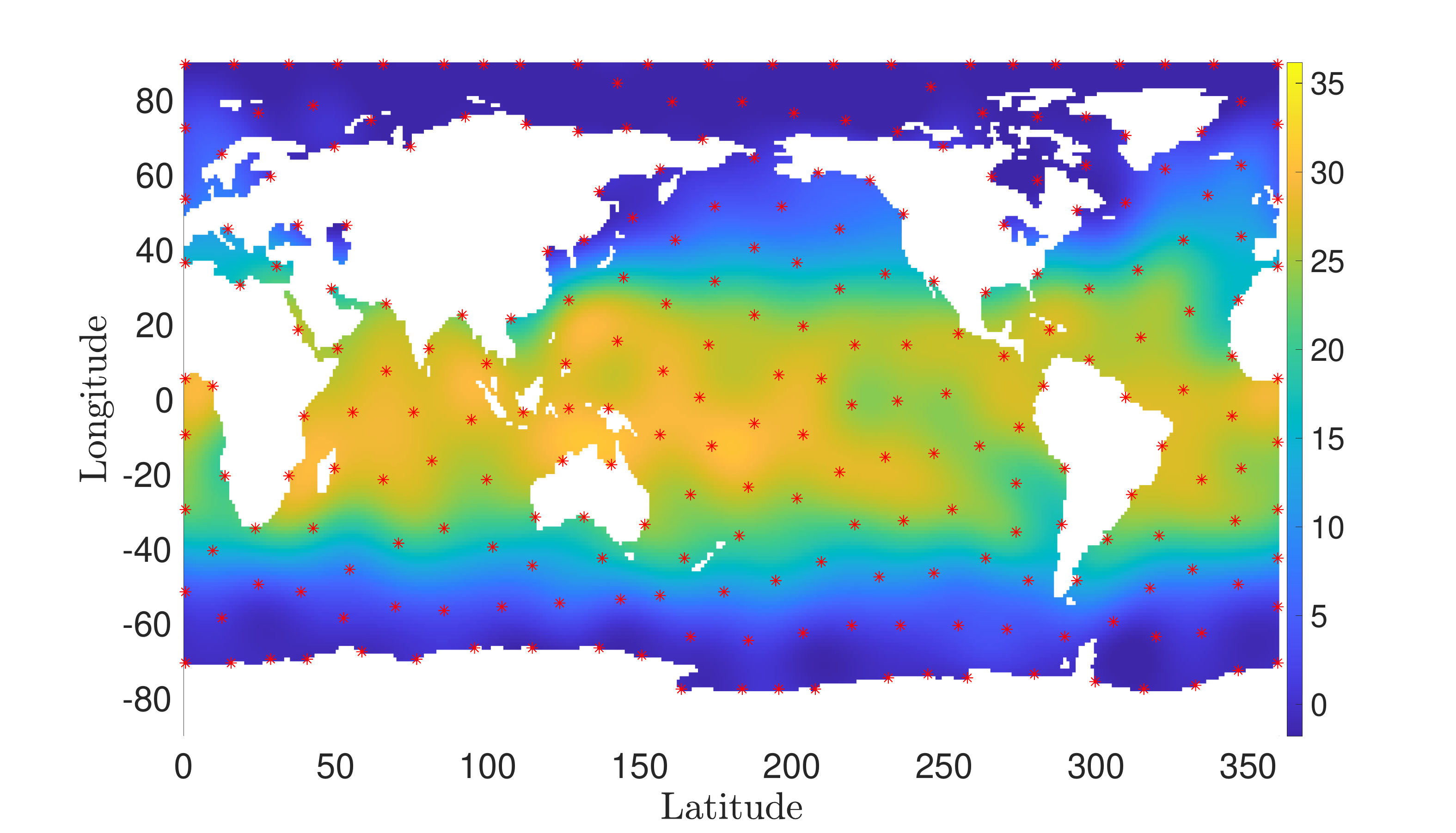}
         
     \end{subfigure}
     \hfill
     \begin{subfigure}[b]{0.49\textwidth}
         \centering
         \includegraphics[width=\textwidth]{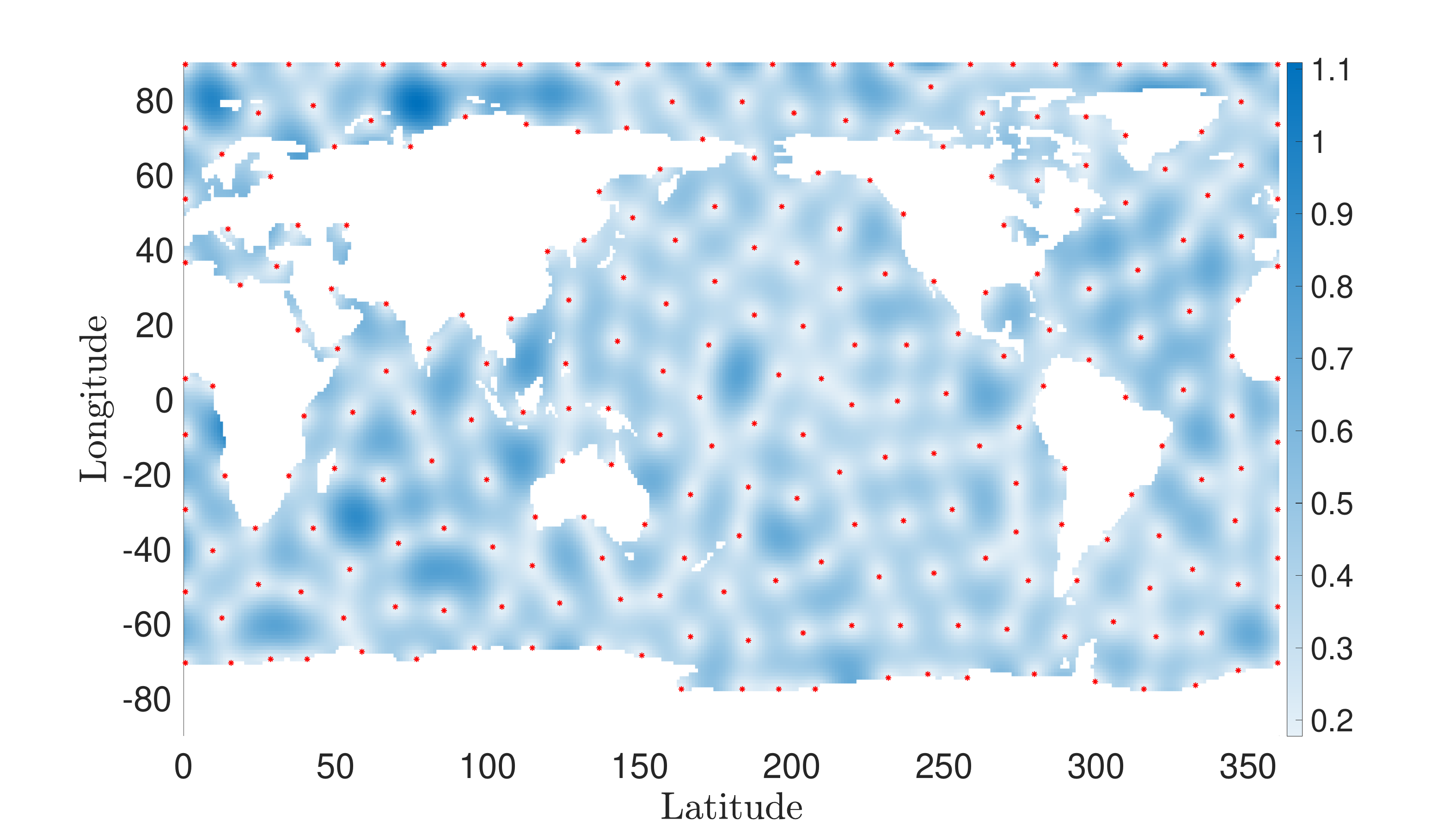}
     \end{subfigure}
        \caption{GP mean prediction (left) and variance (right) plot using the NysGKS (\Cref{alg:nys}). Selected sensor locations are indicated by red asterisks.}
        \label{fig:2drecon}
\end{figure}

The D-optimality of all the Nystr\"om approximation based placement methods are compared to $10,0000$ realizations of randomly placed sensors and the efficient greedy \cite[Algorithm 1]{chen2018fast} selection. The  NysGKS, RPCholesky, and greedy pivoted Cholesky selection algorithms have a much higher D-optimality score than 10,000 realizations of random selection. The greedy pivoted Cholesky selection achieves the highest D-optimality score at 1076.6. This is followed by the rank $k+10$ NysGKS Algorithm (\Cref{alg:nys}), greedy pivoted Cholesky Algorithm, and rank $k$ RPCholesky (\Cref{alg:cholgks}) with D-optimality scores of 1074.8, 1068.4, 1067.5, respectively. 
\begin{figure}[!ht]
    \centering
    \includegraphics[width=0.75\linewidth]{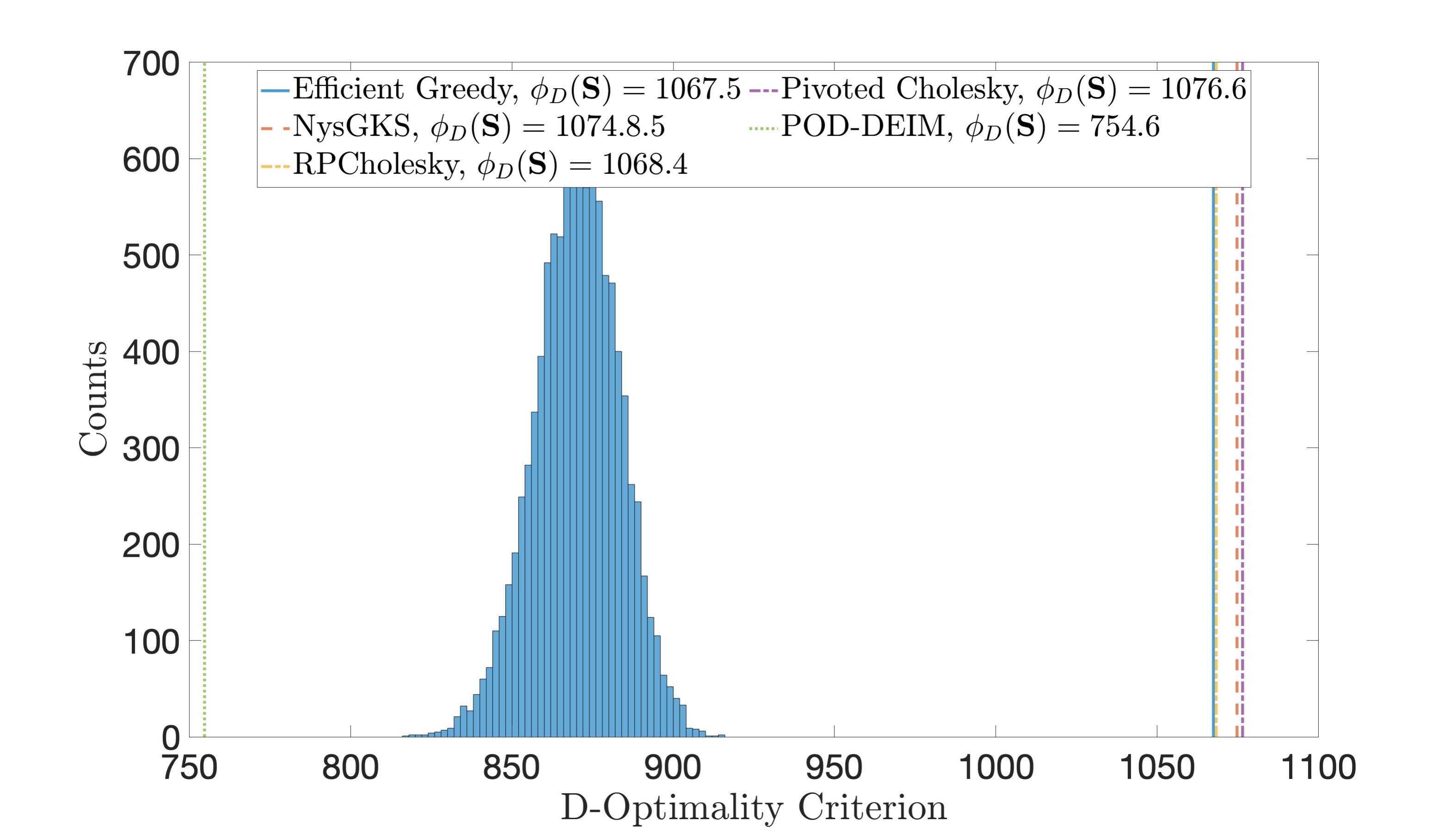}
    \caption{Efficient Greedy Selection \cite[Algorithm 1]{chen2018fast},  NysGKS (\Cref{alg:nys}),  RPCholesky, and greedy pivoted Cholesky (\Cref{alg:cholgks}) compared with 10,000 realizations of random sensor placement in terms of D-optimality.}
    \label{fig:rand_2d}
\end{figure}

\subsubsection{Comparison with POD-DEIM}

The POD-DEIM algorithm  \cite{poddeim1,poddeim2,qdeim} has been used to greedily place sparse sensors for full-state reconstruction and has connections to the D-optimality criterion. 
We train the POD-DEIM selection on sea surface temperature data using data from five years prior to the week of February 4, 2018. We then use $m = 250$ modes to place $k = 250$ sensors and perform a reconstruction for the week of February 4, 2018, shown in \Cref{fig:deim_mr250}. This method yields a relative error  of 0.0680 which is comparable to the relative error of 0.0781 from the NysGKS selected sensors. While the reconstruction error in the POD-DEIM method is slightly lower, the amount of required data in the training step is much larger and more demanding than the zero data requirement in the NysGKS algorithm. Additionally, because the structure of POD-DEIM reconstruction is deterministic, there is a lack of a notion of uncertainty in reconstruction. We see from \Cref{fig:deim_mr250} (left) that the D-optimality of the POD selected sensors is 754.6, which is lower than the D-optimality of the NysGKS selection (1074.8) which indicates that the POD selected sensors has a lower expected information gain.
\begin{figure}[!ht]
    \centering
    \includegraphics[width=0.45\linewidth]{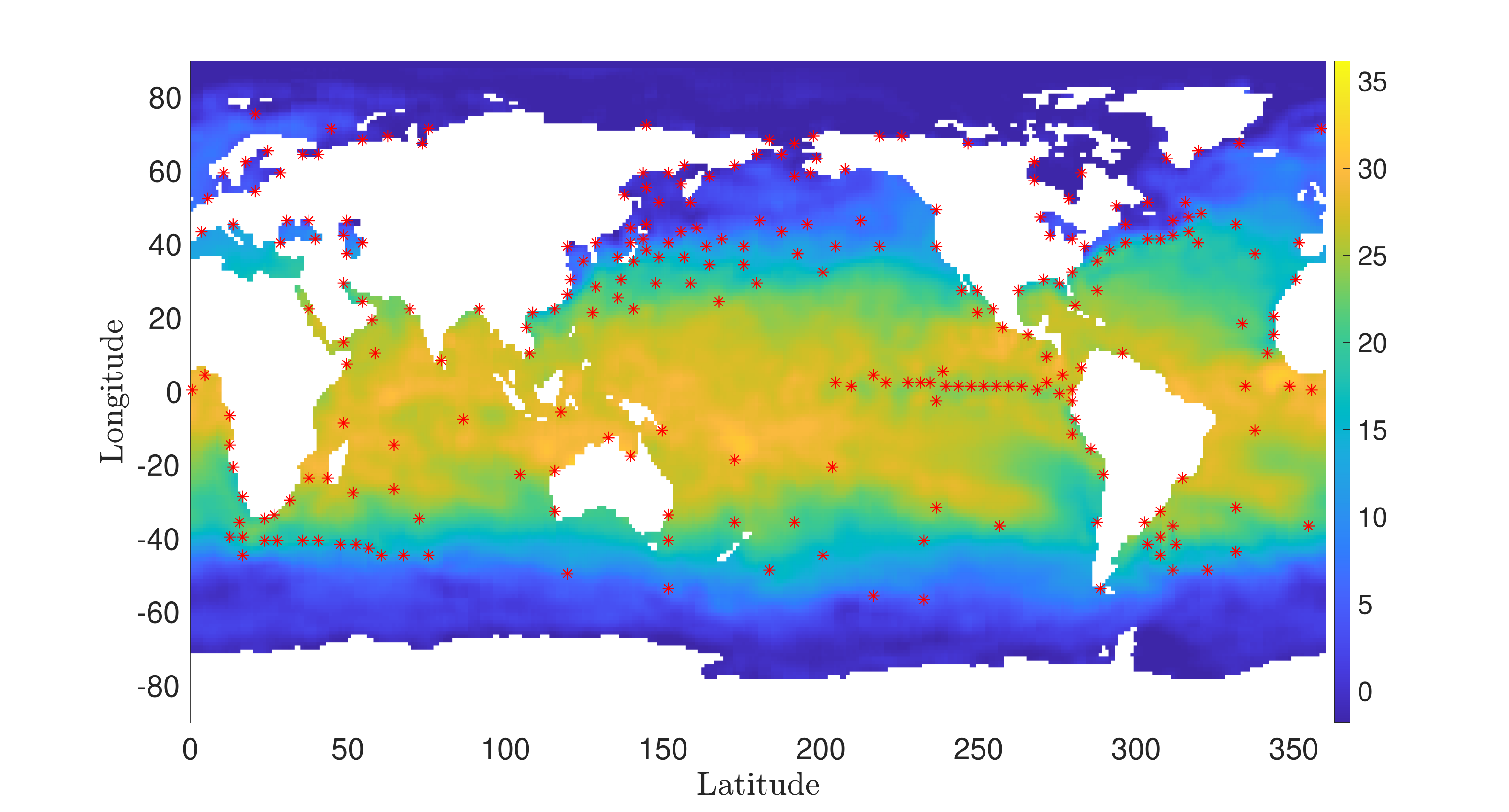}
    \includegraphics[width=0.45\linewidth]{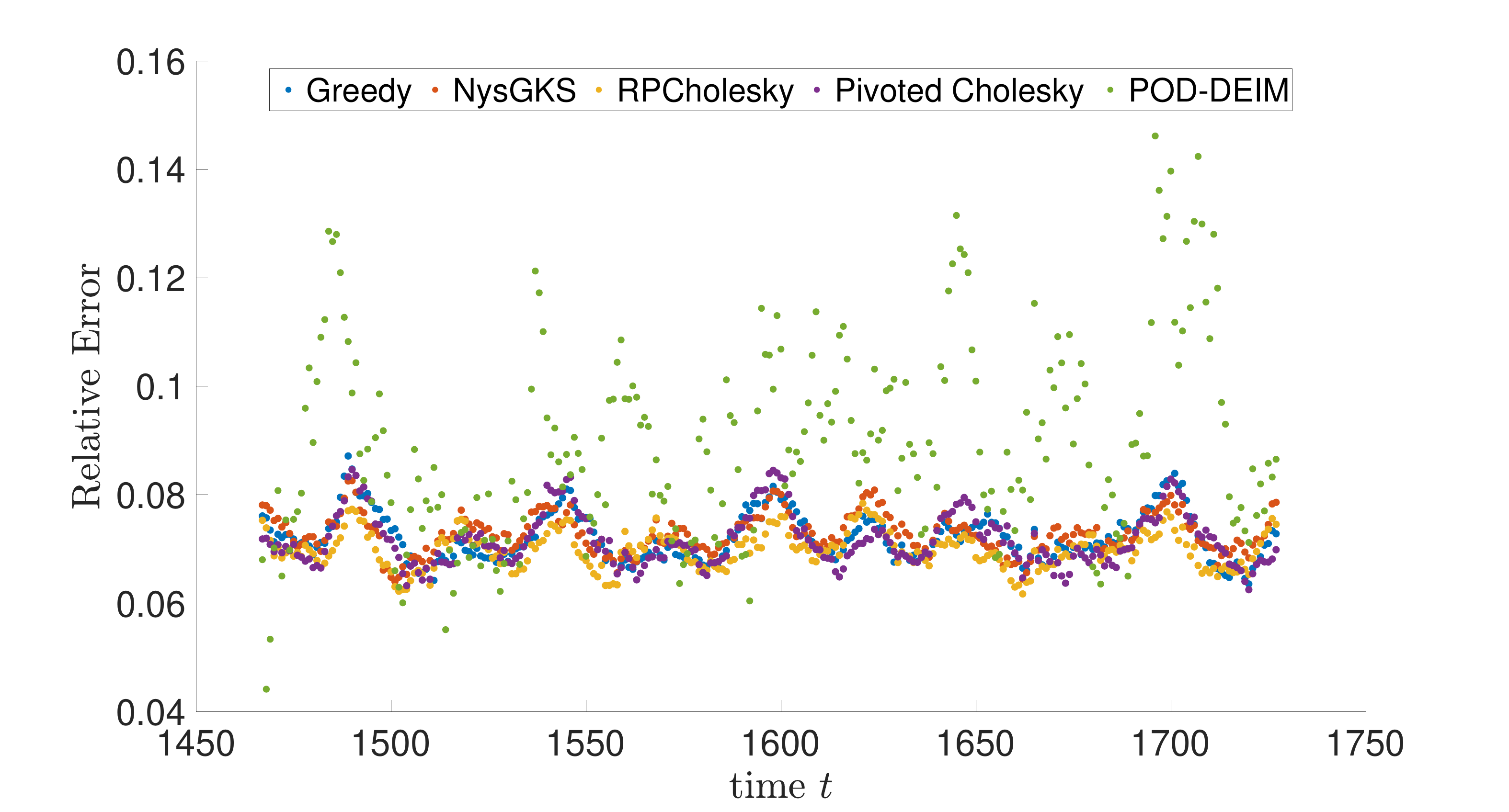}
    \caption{(left) POD Reconstruction with $m=250$ modes and $k = 250$ sensors. (right) Relative Error in Reconstruction as $t$ increases.}
    \label{fig:deim_mr250}
\end{figure}

Next, we consider the relative error of the GP prediction over time using the different selection approaches. Here, we report results over the span of 5 years in one week increments. These snapshots are not included in the training set of POD-DEIM. In \Cref{fig:deim_mr250} (right), we see that although the POD-DEIM approach has a lower relative error of 4.2\% at the start of the time interval, the relative error increases up to about 15\% and its error oscillates with each year of the time interval. This shows that POD-DEIM based sensor placement can become less accurate as time increases, whereas our proposed methods maintain a much lower relative error as time increases, suggesting longer term reliability of sensor placements.

\section{Proofs}\label{sec:proofs}
This section houses the proofs to the theorems presented in the Methods section (\Cref{sec:methods}). \Cref{ssec:background} covers the remaining background needed for the proofs of \Cref{thm1,thm2,thm3,thm:greedypiv}. 

\subsection{Additional Background}\label{ssec:background}
In this subsection, we state facts that are used within the proofs found throughout \Cref{sec:proofs}.
\paragraph{Cauchy's Interlacing Theorem}\label{cauchyinterlacing}
\cite[Theorem 4.3.28]{horn} \cite[Theorem III.1.5]{bhatia} Let $\mat A \in \reals^{n \times n}$ be a symmetric matrix. Let $\mat B \in \reals^{m \times m}, m<n$, be a principal submatrix of $\mat A$. Then, for $ 1 \le i \le  m$, the following inequality holds
        \begin{align*}
                \lambda_i(\mat A) \geq \lambda_i(\mat B) \geq \lambda_{i + n - m}(\mat A).
        \end{align*}
        As noted in \cite[Corollary III.1.5]{bhatia}, if $m = n-1$, then 
\begin{align} \label{eigbounds2}
    \lambda_1(\mat A) \geq \lambda_1(\mat B) \geq \lambda_2(\mat A) \geq \dots \geq \lambda_{n-1}(\mat B) \geq \lambda_n(\mat A).
\end{align}

\paragraph{Loewner partial ordering}
Let $\mat{A}, \mat{B} \in \reals^{n \times n}$ be symmetric matrices. We say $\mat{A} \preceq \mat{B}$ if $\mat{B} - \mat{A}$ is positive semidefinite. This defines a partial ordering on the space of symmetric matrices called the Loewner partial ordering. Properties of this ordering can be found in~\cite[Chapter 7.7]{horn}, which we recapitulate here. If $\mat{A}\preceq \mat{B}$ and $\mat{X} \in \reals^{m\times n}$, then $\mat{XAX}\t \preceq \mat{XBX}\t$. By Weyl's inequality, we have 
\begin{equation}\label{eqn:weylconsq}  \lambda_i(\mat{A}) \le \lambda_i(\mat{B}) \qquad 1 \le i \le n. \end{equation} If $\mat{0} \preceq \mat{A} \preceq \mat{B}$, then $(\mat{I}+\mat{A})^{-1} \succeq (\mat{I}+\mat{B})^{-1}$ and 
\begin{equation}\label{eqn:loewnerinter}
\mat{X}(\mat{I}+\mat{A})^{-1}\mat{X}\t \succeq \mat{X}(\mat{I}+\mat{B})^{-1}\mat{X}\t.
\end{equation}

We present the following lemma, which will be used repeatedly. 
\begin{lemma} \label{matbounds}
Consider the matrix product $\mat A \mat B$ for $\mat A \in \reals^{m \times n}$ and $\mat B \in \reals^{n \times \ell}$. Then 
$$ (\mat{AB})\t (\mat{AB}) \preceq  \|\mat{A}\|_2^2 \mat{B}\t\mat{B}. $$
Furthermore, for $1 \leq i \leq \min\{m,\ell\}$,
\begin{align} \label{eiglemma}
\lambda_i( \mat B\t \mat A\t \mat A \mat B ) \le & \> \|\mat{A}\|_2^2 \lambda_i(\mat B\t \mat B) \\ \label{eiglemma2}
\lambda_i( \mat B\t \mat A\t \mat A \mat B )   \leq & \>  \|\mat B\|_2^2 \lambda_i(\mat A\t \mat A)  .
\end{align}
\end{lemma}

\begin{proof}
The properties of Loewner ordering ensure the majorizing relations $\mat{A}\t\mat{A} \preceq \|\mat{A}\|_2^2 \mat{I}$ and $(\mat{AB})\t \mat{AB} \preceq \|\mat{A}\|_2^2 \,\mat{B}\t\mat{B}$. Then,~\eqref{eiglemma} follows from~\eqref{eqn:weylconsq}. For the second set of inequalities, use $\lambda_i(\mat B\t \mat A\t \mat A \mat B) = \lambda_i( \mat A \mat B \mat B\t \mat A\t)$ and argue as before. 
\end{proof}

\paragraph{Orthogonal Projectors}
 Orthogonal projectors are square matrices that are both symmetric and idempotent. Let $\mat P \in \reals^{n \times n}$ be an orthogonal projector. Let $\vec x \in \reals^n$ be any vector. Then $\vec x\t \vec P \vec x = \vec x\t \vec P^2 \vec x = \vec x\t \vec P\t \vec P \vec x = \| \mat P \vec x \|_2 \geq 0$. Thus, the matrix $\mat P$ is symmetric positive semi-definite. Since $(\mat I_n - \mat P)\t = \mat I_n - \mat P$ and $(\mat I_n - \mat P)^2 = (\mat I_n - \mat P)(\mat I_n - \mat P) = \mat I_n - 2 \mat P + \mat P^2 = \mat I_n - \mat P$, the matrix $\mat I_n - \mat P$ is also an orthogonal projector. Thus, we may write $\mat 0 \preceq \mat P \preceq \mat I$.

\paragraph{Additional Lemmas} We collect several lemmas that will be used in our analysis.

\begin{lemma} \cite[Problem III.6.2]{bhatia}
    For matrices $\mat A \in \reals^{m \times n}$ and $\mat B \in \reals^{n \times p}$, the singular values $\sigma_p(\mat {AB}) \leq \sigma_p(\mat A)\sigma_1(\mat B)$ for $1 \leq p \leq \min\{m,n\}$. \label{fact2}
\end{lemma}

\begin{lemma} 
     \label{frob} For a matrix $\mat A \in \reals^{n \times n}$, 
     we have
     $\|\mat A\|_2^2 \leq \|\mat A\|^2_F \leq n \max_{1\leq j \leq n}\|\vec a_j\|^2_2$.

     \end{lemma}
    \begin{proof}
        The first inequality is a well-known fact so we proceed to showing the second inequality:
        \begin{align*}
            \|\mat A\|_F^2 = \sum_{j= 1}^n \|\vec{a}_{j}\|_2^2 \leq n \max_{1\leq j \leq n}\|\vec a_j\|^2_2.
        \end{align*}
    \end{proof}

\begin{lemma}\cite[Lemma 2.1]{drmac2018discrete}\label{lem:srrqrbounds}
    Let $\mat V_k \in \reals^{n \times k}$ such that $\mat V_k\t\mat V_k = \mat I_k$. Using \cite[Algorithm 4]{srrqr} on $\mat V_k$ with the target rank set to $k$ and $f \geq 1$ results in a selection matrix $\mat S$ such that
    \begin{align*}
        \frac{1}{\sqrt{1 + f^2k(n-k)}} \leq \sigma_j(\mat V_k\t \mat S)\leq 1, \quad 1\leq j \leq k.
    \end{align*}
    Furthermore,
    $1 \leq \|[\mat V_k\t\mat S]^{-1}\|_2\leq \sqrt{1 + f^2k(n-k)}.$
\end{lemma}

\subsection{Bounds for Conceptual GKS}\label{ssec:gksproof}
\begin{proof}[Proof of \Cref{thm1}] 
The proof is similar to that of~\cite[Theorem 3.2]{eswar2024bayesian}. We prove the sequence of inequalities starting from the left to right.

The inequalities $\lambda_i(\mat K) \geq \lambda_i(\mat S^\top \mat K \mat S)$ for $1 \le i \le k$ follow from Cauchy's interlacing theorem \eqref{cauchyinterlacing} on the principal submatrix  $\mat S^\top \mat K \mat S$. 
Since the $\log$ function is monotonic, we have
\begin{align*}
     \ld(\mat I_k + \eta^{-2} \mat \Lambda_k) & = \sum_{i=1}^{k} \log(1+\eta^{-2}\lambda_i(\mat K))\\ & \geq \sum_{i=1}^{k} \log(1+\eta^{-2}\lambda_i(\mat S^\top \mat K \mat S))  = \ld(\mat I_k +\eta^{-2}\mat S^\top \mat K \mat S).
\end{align*}
Since $\mat{S}$ is an arbitrary selection operator, we have shown $\ld(\mat I_k + \eta^{-2} \mat \Lambda_k) \ge \phi_D(\mat{S}_{\rm opt})$.

The second set of inequalities follow from the definition of $\mat{S}_{\rm opt}$.
For the final set of inequalities, write 
\[ \mat\Lambda_k = [\mat V_k\t\mat S]^{-\top} [\mat V_k\t\mat S]\t \mat \Lambda_k^{1/2} \mat \Lambda_k^{1/2} [\mat V_k\t\mat S][\mat V_k\t\mat S]^{-1},\]
and apply~\eqref{eigbounds2} to get for $1 \le i \le k$
\begin{align*}
    \lambda_i(\mat \Lambda_k) \leq \|[\mat V_k\t\mat S]^{-1}\|_2^2 \lambda_i(\mat S\t \mat K \mat S).
\end{align*}
Using this inequality, we can now bound $ \ld(\mat I_k + \eta^{-2}\mat S^\top \mat K \mat S) $ from below. That is, 
\begin{align*}
    \ld(\mat I_k + \eta^{-2}\mat S^\top \mat K \mat S) & = \sum_{i=1}^{k} \log(1+\eta^{-2}\lambda_i(\mat S^\top \mat K \mat S)) \\ & \geq  \sum_{i=1}^{k} \log \left(1+\eta^{-2}\frac{\lambda_i(\mat K)}{||[\mat V_k^\top \mat S]^{-1}||^2_2} \right)  = \ld\left(\mat I_k + \frac{\eta^{-2} \mat \Lambda_k}{||[\mat V_k^\top \mat S]^{-1}||^2_2}\right).
\end{align*} 
Identify $\phi_D(\mat{S}) = \ld(\mat I_k + \eta^{-2}\mat S^\top \mat K \mat S)$ to complete the proof. 
\end{proof}

\subsection{Bounds on D-optimality Criterion with Nystr\"om Approximation}\label{ssec:nysproof}

\begin{proof}[Proof of \Cref{thm3}] The first inequality is obvious from the definition of $\mat{S}_{\rm opt}$. For the next inequality, note that $\mat{K} \succeq \wh{\mat{K}}$ so that $\mat{S}\t\mat{K} \mat{S} \succeq \mat{S}\t\wh{\mat{K}} \mat{S}$. So, it follows from the properties of log-determinant and~\eqref{eqn:weylconsq} that 
\[ \phi_D(\mat{S}) = \ld(\mat{I}+ \eta^{-2}\mat{S}\t\mat{K} \mat{S}) \ge \ld(\mat{I}+\eta^{-2}\mat{S}\t\wh{\mat{K}} \mat{S}). \]
Applying \Cref{thm1} to the matrix $\wh{\mat{K}}$ and expanding the log-determinant, we get 
\begin{equation}\label{eqn:inter} \phi_D(\mat{S}) \ge \sum_{i=1}^k\log\left(\mat{I} +  \frac{\eta^{-2}\lambda_i(\wh{\mat{K})}}{\|[\wh{\mat{U}}_k\t{\mat{S}}]^{-1}\|_2^2}\right) \ge \sum_{i=1}^k\log\left(\mat{I} +  \frac{\eta^{-2}\lambda_i(\mat{K})}{\|[\wh{\mat{U}}_k\t{\mat{S}}]^{-1}\|_2^2 (1 +  \beta \lambda_i(\mat \Lambda_k^{-1}))}\right) ,\end{equation}
where $\beta:=\|\mat \Lambda_\perp^{1/2} \mat \Omega_{n-k}\mat \Omega_k^\dagger\|_2^2$. In the last step, we have used the inequalities in \Cref{thm2}.

We now place probabilistic bounds on $\beta$. We have assumed that $\mat \Omega \in \reals^{n \times p}$ is a Gaussian matrix with $p \geq 4$ and $t,u\geq1$. Using the intermediate results found in \cite[Section 10.3, Theorem 10.8]{HMT}, we have with probability at most $2t^{-p} + e^{-u^2/2}$,
\begin{align*}
    \|\mat \Lambda_\perp^{1/2} \mat \Omega_{n-k} \mat \Omega_k^\dagger\|_2 > \|\mat \Lambda_\perp^{1/2}\|_2 \left( \sqrt{\frac{3k}{p+1}} \cdot t +  \frac{e\sqrt{k+p}}{p+1} \cdot ut\right) + \|\mat \Lambda_\perp^{1/2}\|_F \frac{e\sqrt{k+p}}{p+1} \cdot t. 
\end{align*}
Note that $\|\mat \Lambda_\perp^{1/2}\|_2 = \|\mat \Lambda_\perp\|_2^{1/2}$ and $\|\mat \Lambda_\perp^{1/2}\|_F = \sqrt{\trace(\mat\Lambda_\perp)}$. Let $t=e,u=\sqrt{2p}, p\geq 4$ similar to \cite[Section 10.3, Corollary 10.9]{HMT}. 
Thus, with probability at least $1-3e^{-p}$,
\begin{align*}
    \beta = \|\mat \Lambda_\perp^{1/2} \mat \Omega_{n-k} \mat \Omega_k^\dagger\|_2^2 \leq \left[  \|\mat \Lambda_\perp\|_2^{1/2} \left( 16 \sqrt{1+ \frac{k}{p+1}}\right) + \sqrt{\trace(\mat \Lambda_\perp)} \frac{8\sqrt{k+p}}{(p+1)} \right]^2 =: \gamma.
\end{align*}
Next, from \Cref{lem:srrqrbounds}, we obtain the deterministic bound $\|\wh{\mat{V}}_k\t{\mat{S}}]^{-1}\|_2^2 \le 1 + f^2k(n-k)$. Plug this inequality as well as the bound on $\beta$ and rearrange to get the final bound. 
\end{proof}

\subsection{Bounds on greedy pivoted Cholesky decomposition}\label{ssec:choleskyeigs}
\begin{proof}[Proof of \Cref{thm:greedypiv}]
 The proof proceeds in several steps.

\paragraph{Step 1: Summary of complete pivoting} The greedy pivoted is nothing but complete pivoting. Thus, we summarize some results from complete pivoting that will be useful in our analysis. Consider the positive semidefinite matrix $\mat K$. Then by \cite[Theorem 10.9]{higham2002accuracy}, there exists a permutation matrix $\mat \Pi$ such that
\[ \mat{A} := \mat{\Pi}\t\mat{K\Pi} = \bmat{ \mat{A}_{11} & \mat{A}_{12} \\ \mat{A}_{12}\t &\mat{A}_{22}} = \mat{R}\t \mat{R} \] 
where $\mat\Pi$ is the permutation matrix obtained by interchanges in complete pivoting and $\mat{R}$ is the (upper triangular) Cholesky factor. By~\cite[Equation 10.13]{higham2002accuracy}, the following bounds apply to the Cholesky factor
\begin{equation}\label{eqn:choleskyentries}r_{ss}^2 \ge \sum_{i=s}^{\min\{j,t\}}r_{ij}^2, \qquad s+1\le j \le n, 1 \le s \le t, \end{equation}
where $t = \rank(\mat{A})$. In particular, this implies that the diagonal elements $r_{11} \geq r_{22} \geq \dots \geq r_{tt}$, and for the off-diagonal elements
\[ r_{ss} \ge |r_{sj} | \qquad 1 \le s \le t, s \le j \le n.\]
Note that for $s >  t $, $r_{sj} = 0$ for $1 \le j \le n$.
\paragraph{Step 2: Nystr\"om approximation} Since the matrix $\mat A$ and $\mat K$ are orthogonally similar, we use the Nystr\"om approximation of $\mat A$ to derive bounds on the eigenvalues of $\mat K$. The Nystr\"om approximation for $\mat{A}$ is 
\[ \mat{A} \approx \bmat{\mat{A}_{11} \\ \mat{A}_{12}\t} \mat{A}_{11}^{-1} \bmat{\mat{A}_{11} & \mat{A}_{12}} = \bmat{\mat{A}_{11} & \mat{A}_{12} \\ \mat{A}_{12}\t & \mat{A}_{12}\t\mat{A}_{11}^{-1} \mat{A}_{12} }. \] 
Thus, the Nystr\"om approximation to $\mat{K}$ is 
\[ \wh{\mat{K}} = \mat\Pi\bmat{\mat{A}_{11} & \mat{A}_{12} \\ \mat{A}_{12}\t & \mat{A}_{12}\t\mat{A}_{11}^{-1} \mat{A}_{12} }\mat\Pi\t. \]
We can express this as $\wh{\mat{K}} = \mat{FF}\t$, where 
\[ \mat{F} = \mat\Pi\bmat{\mat{A}_{11} \\ \mat{A}_{12}\t} \mat{R}_{11}^{-\top},\]
where $\mat{R}_{11} \in \reals^{k\times k}$ is the leading principal submatrix of $\mat{R}$. In \cref{alg:cholgks}, the GKS approach is applied to $\wh{\mat{K}} $. Thus, from the proof of \cref{thm1}, 
\[ \lambda_i(\mat{S}\t\wh{\mat{K}}\mat{S}) \ge \frac{\lambda_i(\wh{\mat{K}})}{\|[\mat{V}_k\t\mat{S}]^{-1}\|_2^2}, \qquad 1 \le i \le k , \]
where $\mat{V}_k$ are the right singular vectors of $\mat{F}$ and thus, $\wh{\mat{K}}$. By using the orthogonal similarity and Cauchy interlacing theorem, we can write 
\[ \lambda_i(\mat{S}\t\wh{\mat{K}}\mat{S}) \ge \frac{\lambda_i(\mat{A}_{11})}{\|[\mat{V}_k\t\mat{S}]^{-1}\|_2^2}, \qquad 1 \le i \le k.   \]

\paragraph{Step 3: Bounds for $\mat{A}_{11}$} Let $\mat{R}_{1:i}$ represent the upper triangular leading principal submatrix of $\mat{R}$ for $1\le i\le k $. Thus, by~\cite[Theorem 8.14]{higham2002accuracy}, 
\[ \frac{1}{r_{ii}}\le \|\mat{R}_{1:i}^{-1}\|_2 \le \frac{2^{i-1}}{r_{ii}} \qquad 1 \le i \le k. \]
By Cauchy interlacing theorem, 
\[ \lambda_i(\mat{A}) \ge \sigma_i(\mat{R}_{1:i})^2 \ge \frac{r_{ii}^2}{4^{i-1}} \qquad 1 \le i \le k. \]
Next, we provide lower bounds for $r_{ii}$, following the strategy in~\cite[Theorem 6.31]{lawson1995solving}. Consider a block matrix of $\mat R$ of size $(n+1-i) \times (n+1-i)$ consisting of the last $(n+1-i)$ rows and columns of $\mat R$ and call this $\mat R_{i:n}$. Note that the largest element in magnitude of $\mat R_{i:n}$ is $|r_{ii}|$, and we can partition 
    \[ \mat{A} = \bmat{* & * \\ * & \mat{R}_{i:n}\t\mat{R}_{i:n}} \qquad 1\le i \le k. \]
    We may apply Cauchy interlacing theorem to the sequence of matrices $\{\mat{R}_{i:n}\}_{i=1}^k$
    \begin{align*}
          \lambda_i(\mat A) =\sigma^2_i(\mat R_{1:n}) \leq \sigma^2_{i-1}(\mat R_{2:n}) \leq \dots \leq \sigma^2_{1}(\mat R_{i:n}) = \|\mat R_{i:n}\|_2.
    \end{align*}
    Applying \Cref{frob} to $\mat R_{i:n}$, we have 
    \begin{align*}
        \|\mat R_{i:n}\|_2 \leq (n+1-i)^{1/2} \max_{1 \leq j \leq n-i+1} \|\mat{R}_{i:n}\vec{e}_j\|_2 = (n+1-i)^{1/2}|r_{ii}|,
    \end{align*}
    where in the last step, we have used~\eqref{eqn:choleskyentries}. Thus, we have 
    \[ \lambda_i(\mat{A}) \ge \lambda_i(\mat{A}_{11}) \ge \frac{\lambda_i(\mat{A})}{4^{i-1}(n-i+1)} \qquad 1 \le i \le k. \] 
\paragraph{Step 4. Finishing the proof}
Combining the intermediate results from Steps 2 and 3, 
\[ \lambda_i(\mat{S}\t\wh{\mat{K}}\mat{S}) \ge \frac{\lambda_i(\mat{A})}{4^{i-1}(n-i+1)(1 + f^2k(n-k))}, \qquad 1 \le i \le k.   \]
In the last step, we have also used \cref{lem:srrqrbounds}. Once again using Cauchy interlacing theorem
\[ \lambda_i(\mat{S}\t{\mat{K}}\mat{S}) \ge \lambda_i(\mat{S}\t\wh{\mat{K}}\mat{S}) \ge \frac{c_{ii}\lambda_i(\mat{K})}{1 + f^2k(n-k)} \qquad 1 \le i \le k, \]
where the diagonal matrix $\mat{C}$ is defined in the statement of the theorem. We have also used the fact that $\mat{A}$ and $\mat{K}$ have the same eigenvalues by similarity. From the eigenvalue bounds, it is easy to deduce the result involving the D-optimality.

\end{proof}

\section{Conclusion}

In this paper, we presented several algorithms for D-optimal sensor placement in Gaussian process regression using column subset selection. By combining numerical linear algebraic approaches to the sensor placement problem, we develop a framework for placing all $k$ sensors at once. We provided various approaches based on the GKS framework. When the covariance matrix is too large to form, we paired the Nystr\"om approximation with the GKS framework. We give theoretical bounds on the D-optimality of the sensor selection when using the conceptual algorithm (\Cref{alg:cssp}), random projection Nystr\"om approximation with GKS selection (\Cref{alg:nys}), and greedy pivoted Cholesky with GKS selection (\Cref{alg:cholgks}). In addition, we provide bounds on the eigenvalues of Nystr\"om approximated matrices through random projection and greedy pivoted Cholesky. Although the asymptotic flop count of efficient greedy is the same as our proposed algorithms, the numerical experiments show that the proposed algorithms perform better than greedy in terms of D-optimality and relative error. The methods in this paper are computationally efficient and have the advantage that, with appropriately chosen hyperparameters, the sensor selection requires no observation data while providing quantifiable uncertainties in its predictions. 

There are future directions, some of which we discuss here. The sensor placements made by our presented algorithms can potentially act as initial placements in exchange algorithms. We can quantify the effect of initialization in the exchange algorithm.  In \Cref{ssec:nys}, a column sampling approach can be used instead of the random Nystr\"om approximation and a pivoted Cholesky factorization to approximate the covariance matrix $\mat K$.  The sensor selections are independent of the kernel and number of dimensions of the candidate sensor data and may allow for a wider breadth of applications of GPs such as surrogate modeling as shown in \Cref{ssec:4d}. While our numerical examples were limited to studying only the square exponential kernel, future work may focus on other kernels, e.g., Mat\'ern kernels.

\printbibliography
\appendix
In \Cref{sec:results}, we give additional results from the experimental settings discussed in the main section. In addition to the 1D and 2D cases, we also present results for a 4D surrogate modeling problem. The additional proof for bounding Nystr\"om approximated eigenvalues is in \Cref{sec:proofs}.

\section{Additional Numerical Experiments} \label{sec:results}

\subsection{Thin Liquid Film Dynamics - Performance of Pivoted Cholesky} \label{ssec:gksvsnogks}

In this experiment, we consider the performance of the pivoted Cholesky algorithms. 
Since the pivots can also be interpreted as sensor placements, we consider the pivoted Cholesky algorithms with and without the inclusion of the GKS step. As a baseline, we also compare these results to the D-optimality of the conceptual GKS selection and 10,000 realizations of random selection. From \Cref{fig:gkscompare}, we see that all the proposed algorithms perform much better than the random designs. Furthermore, we see that the addition of the GKS step after performing a pivoted Cholesky (\Cref{alg:cholgks}) factorization improves the D-optimality scores. Thus, the numerical experiments that follow focus on combining the pivoted Cholesky algorithms with the GKS framework.

\begin{figure}[!ht]
    \centering 
\includegraphics[width=0.75\linewidth]{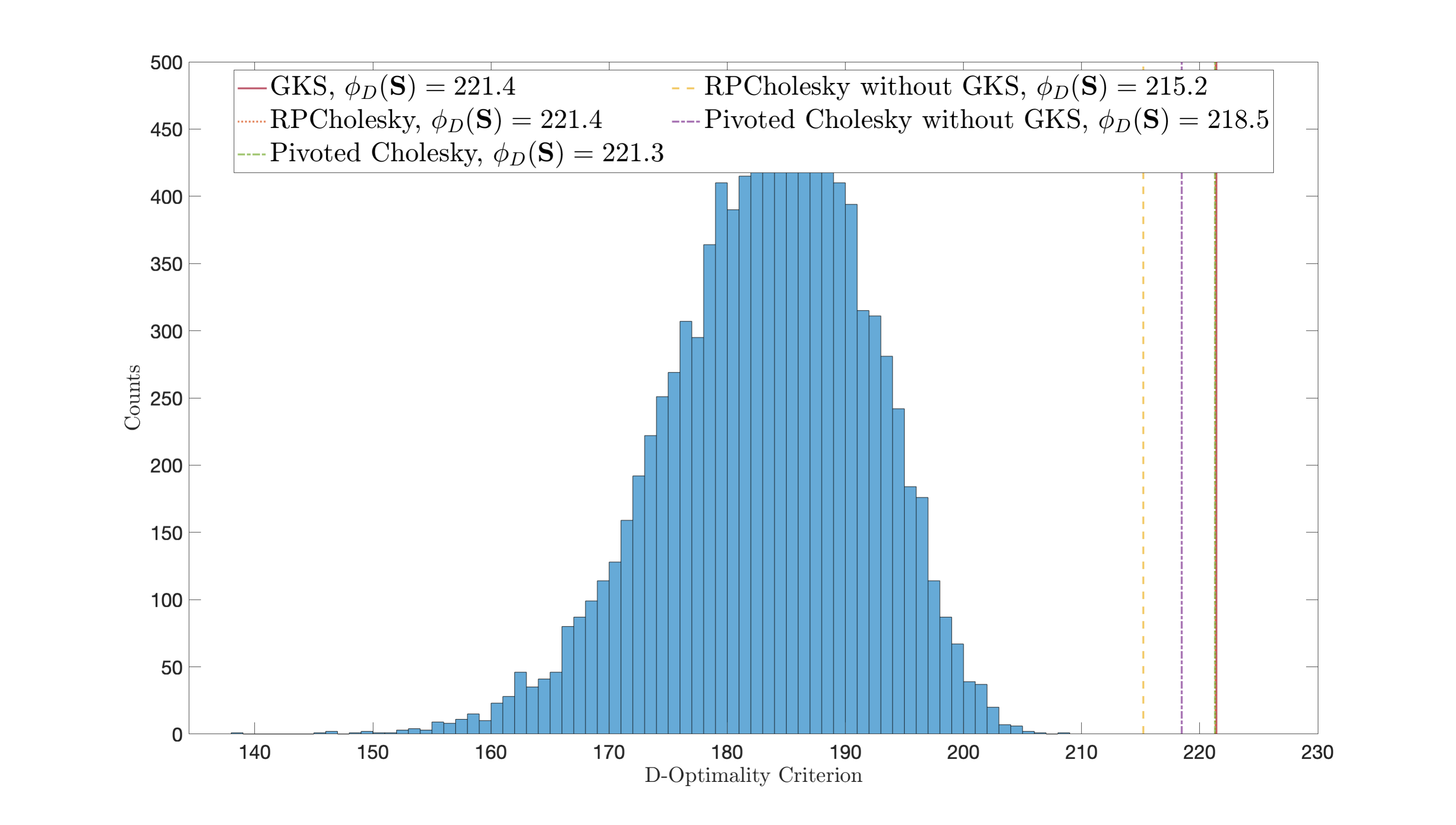}
     
    \caption{Conceptual GKS (\Cref{alg:cssp}) and rank $k$ RPCholesky and greedy pivoted Cholesky (\Cref{alg:cholgks}) with and without GKS compared with 10,000 realizations of random sensor placement in terms of D-optimality.}
    \label{fig:gkscompare}
    \end{figure}

\subsubsection{Error in Reconstruction over Time}

\Cref{tab:statsovert_1d_k50} reports the minimum, maximum, mean, and standard deviation of the relative reconstruction error over the time interval. We observe that overall, each selection method maintains a similar magnitude in relative error \eqref{relerr} over the time interval. However, the greedy selection was higher in terms of the minimum, maximum, mean and standard deviation in relative error \eqref{relerr} over time. 

\begin{table}[!ht]
    \centering
    \begin{tabular}{c|c|c|c|c}
      Algorithm &   Min&   Max&Mean &Std. Dev.\\ \hline
      Conceptual GKS &  $4.4070\e{-4}$& 
      0.0097 & 0.0033 &0.0029\\ 
     \hline
     NysGKS & $4.7016\e{-4}$& 
     0.0094 &  0.0033 &0.0029\\
     \hline
     RPCholesky + GKS  & $4.8037\e{-4}$& 
     0.0096 &0.0034 & 0.0028\\
     \hline
     Greedy pivoted Cholesky + GKS &  $4.5317\e{-4}$& 
   0.0095 &0.0034 & 0.0029\\ \hline
    Efficient Greedy &  $6.1590\e{-4}$ & 
      0.0101 & 0.0035 & 0.0029
     \end{tabular}
    \caption{Summary statistics of the relative reconstruction error for each algorithm over time $0 \le t \le 300$.}
    \label{tab:statsovert_1d_k50}
\end{table}

\subsubsection{Error and D-optimality as $k$ increases}
We now vary the number of sensors $k$ for $k = 20,21,\dots,50$ and examine how algorithm performance depends on $k$.
For this study, sensor data comes from the $t = 140$ snapshot. Here, we compare the efficient greedy \cite[Algorithm 1]{chen2018fast}, conceptual GKS (\Cref{alg:cssp}), NysGKS (\Cref{alg:nys}), RPCholesky and greedy pivoted Cholesky (\Cref{alg:cholgks}) algorithms for sensor selection. \Cref{fig:1D_k_inc} (left) shows an expected general increase in the D-optimality score of each algorithm as $k$ increases. Similarly, as we increase the number of sensors $k$, we see in \Cref{fig:1D_k_inc} (right) that the relative reconstruction error decreases. This is expected since the more available sensors there are for data collection, the more information we can expect to gain. In both figures, the efficient greedy algorithm \cite[Algorithm 1]{chen2018fast} performs slightly worse than the other algorithms in comparison. The D-optimality of the efficient greedy selection trails below the D-optimality of all other methods, and the relative error \eqref{relerr} in reconstruction using the greedy selection tends to be higher than the relative error \eqref{relerr} in reconstruction of all other selection methods.

\begin{figure}[!ht]
\label{fig:relerr_1d} 
     \centering
     \includegraphics[width=\linewidth]{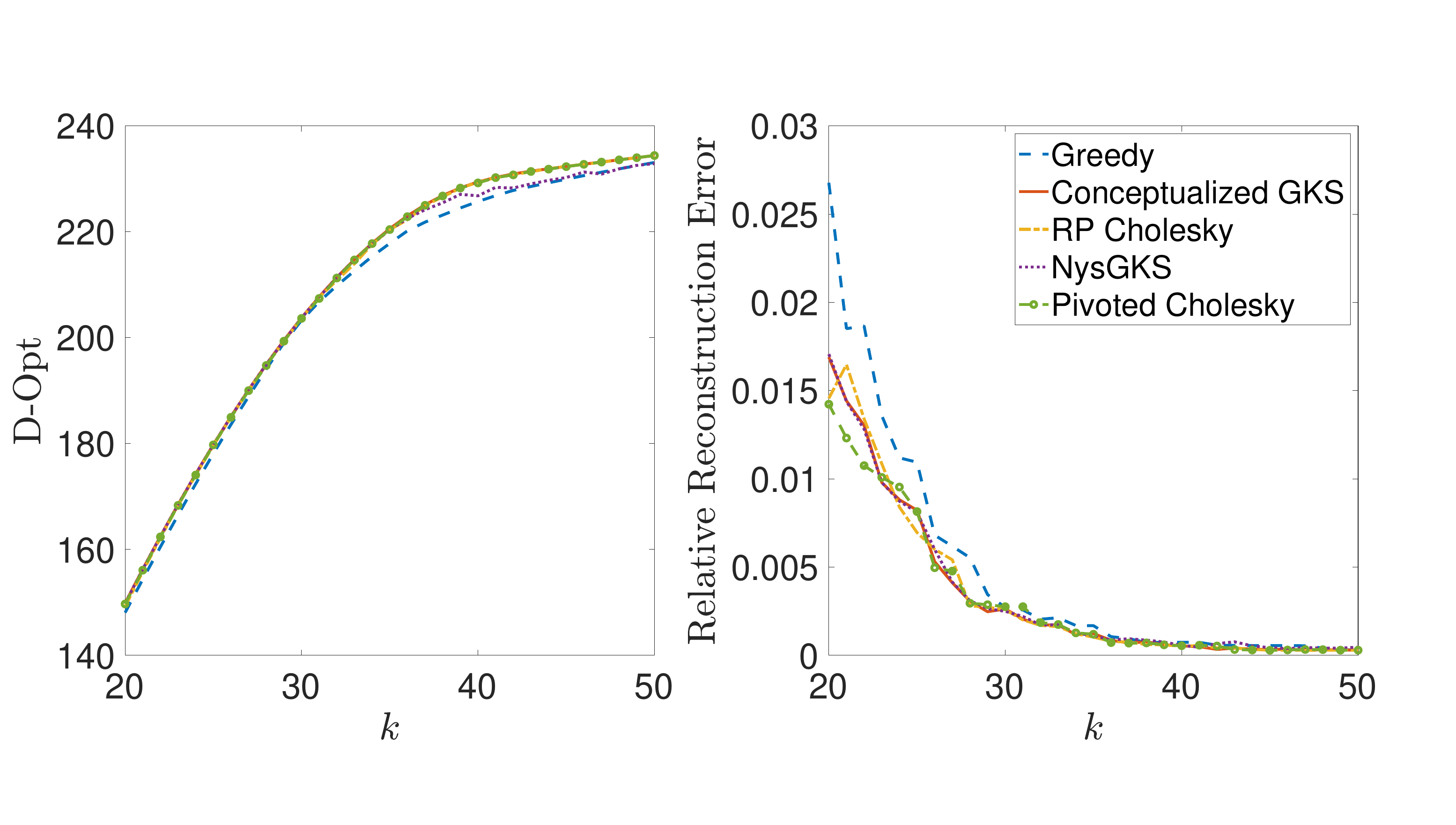}
        \caption{Plots comparing the performance of sensor selection between the conceptual GKS algorithm (\Cref{alg:cssp}), greedy algorithm~\cite[Algorithm 1]{chen2018fast}, rank $k+10$ NysGKS (\Cref{alg:nys}), and rank $k$ RPCholesky and Pivoted Cholesky (\Cref{alg:cholgks}) in terms of D-optimality (left) and relative reconstruction  error (right) as $k$ increases from $k = 20, 21, \dots, 50$.}
        \label{fig:1D_k_inc}
\end{figure}

\subsection{Sea Surface Temperature} \label{ssec:sst}
In this section, we show additional numerical results on the NOAA Optimum Interpolation (OI) SST V2 data, which was provided by the NOAA PSL, Boulder, Colorado, USA, from their website at \url{https://psl.noaa.gov}\cite{sst}. We discuss the effect on the D-optimality and relative reconstruction error \eqref{relerr} as the number of available sensors $k$ increases (see \Cref{ssec:err_dopt_k}), the computational cost of the NysGKS (\Cref{alg:nys}) selection when using H2Pack \cite{h2cite1,h2cite2} in \Cref{costvsn}, and continue the comparison between the POD-DEIM \cite{poddeim1,poddeim2,qdeim} and our methods in \Cref{sssec:anomaly}. Lastly, in \Cref{ssec:geodist} we explore the effect of using the great circle distance in the kernel function instead of Euclidean distance to account for the shape of the earth.

\subsubsection{Error and D-Optimality As \texorpdfstring{$k$}{} increases}  \label{ssec:err_dopt_k}
We report the D-optimality on the week of February 4, 2018 and the average relative error in reconstruction (\ref{relerr}) of the sea surface temperature over 5 years to compare the performance between NysGKS (\Cref{alg:nys}), RPCholesky and greedy pivoted Cholesky (\Cref{alg:cholgks}) placed sensors as the number of sensors varies from $k = 150, 155, \dots, 350$. In \Cref{fig:2D_k_inc} (left), the D-optimality of the sensor placements steadily increases as $k$ increases. The relative reconstruction error steadily decreases in \Cref{fig:2D_k_inc} (right) as the number of available sensors increases.

\begin{figure}[!ht]
    \centering
    \includegraphics[width=\linewidth]{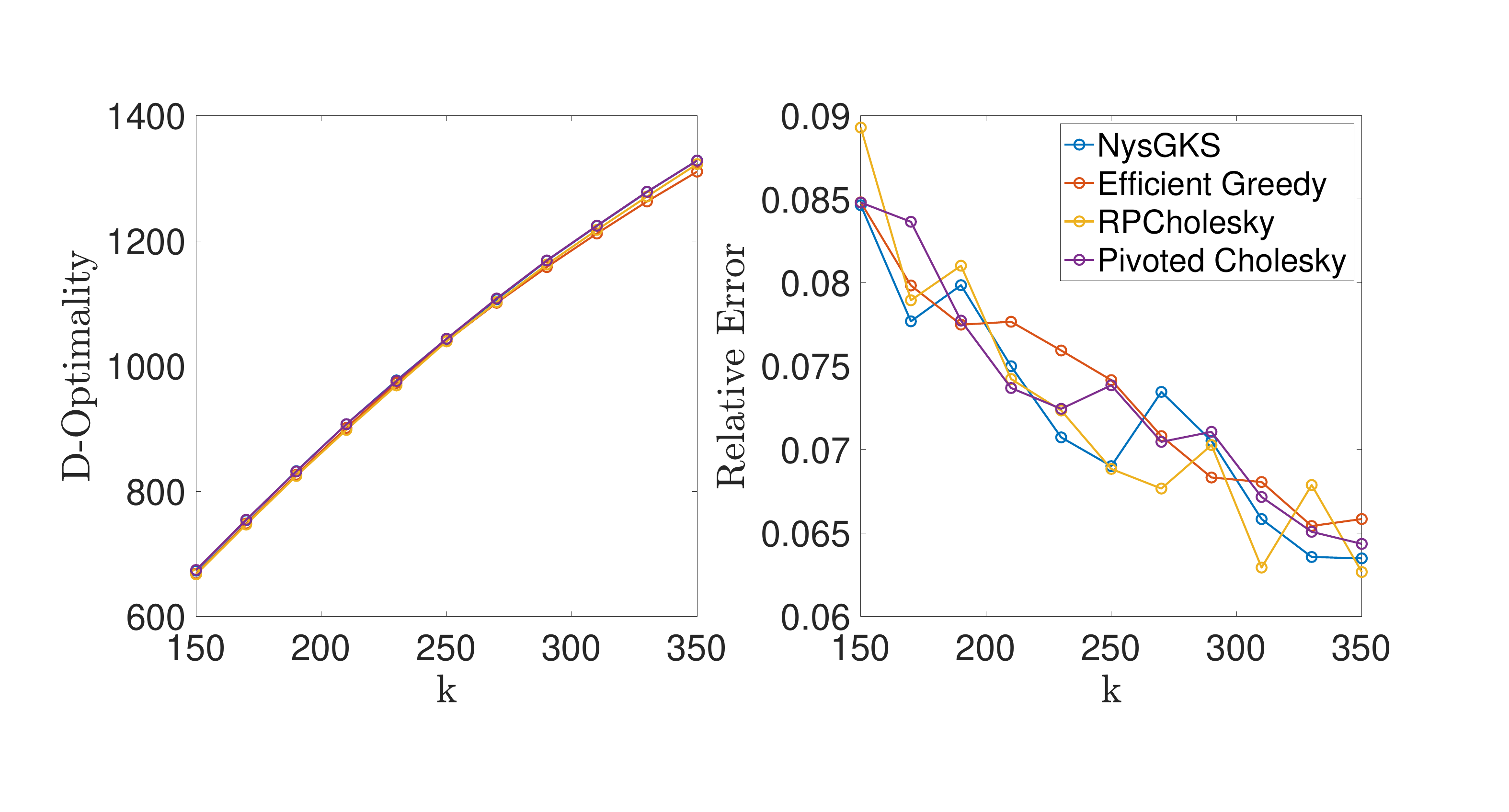}
    \caption{D-Optimality (left) and Relative Reconstruction Error (right) in the NysGKS (\Cref{alg:nys}), efficient greedy \cite[Algorithm 1]{chen2018fast}, RPCholesky and greedy pivoted Cholesky (\Cref{alg:cholgks}) algorithms as $k$ increases.
    }
    \label{fig:2D_k_inc}
\end{figure}

\subsubsection{Cost vs n}\label{costvsn}
In this section, we report the performance of the Nystr\"om approximated GKS algorithm with H2Pack \cite{h2cite1,h2cite2} as the number of candidate sensors $n$ increases. We fix the number of available sensors $k = 250$ and vary the number of candidate sensors $n$ in our sea surface temperature domain. To provide an accurate measure of the computation time, we utilize NC State University's HPC cluster. We acknowledge the computing resources provided by North Carolina State University High Performance Computing Services Core Facility (RRID:SCR\_022168). We use an Intel Xeon based Linux cluster. In \Cref{tab:h2performance}, we observe that as the number of candidate sensors $n$ increases, the rate of the increase in wall clock time is approximately linear.

\begin{table}[!ht]
    \centering
    \begin{tabular}{|c|c|c|c|c|c|}
    \hline
        Number of Candidate Sensors $n$ &  2048 &  4096 & 8192 & 16384 & 32768\\
        \hline
 Computation Time (seconds) & 15.53 &  23.18  & 30.04  & 32.13 &  38.76
\\
        \hline
    \end{tabular}
    \caption{Computational Complexity of the H2Pack Implementation of the NysGKS Algorithm on various problem sizes.}
    \label{tab:h2performance}
\end{table}

\subsubsection{Further Comparison with POD-DEIM}\label{sssec:anomaly}
In this subsection, we compare our methods to the POD-DEIM method further by looking at the relative reconstruction error \eqref{relerr} over time.  In \Cref{fig:err2d} and \Cref{fig:anomaly}, we report results over the span of 5 years in one week increments starting from the week of February 4, 2018. 

\paragraph{Performance of POD-DEIM vs POD}
\Cref{fig:err2d} is the same plot as \Cref{fig:deim_mr250} with the addition of the POD approximation error (red dashed line). We see that the POD approximation error is much lower than the POD-DEIM approximation error signaling that the DEIM step worsens the sea surface temperature reconstruction.

\begin{figure}[!ht]
    \centering
    \includegraphics[width=0.75\linewidth]{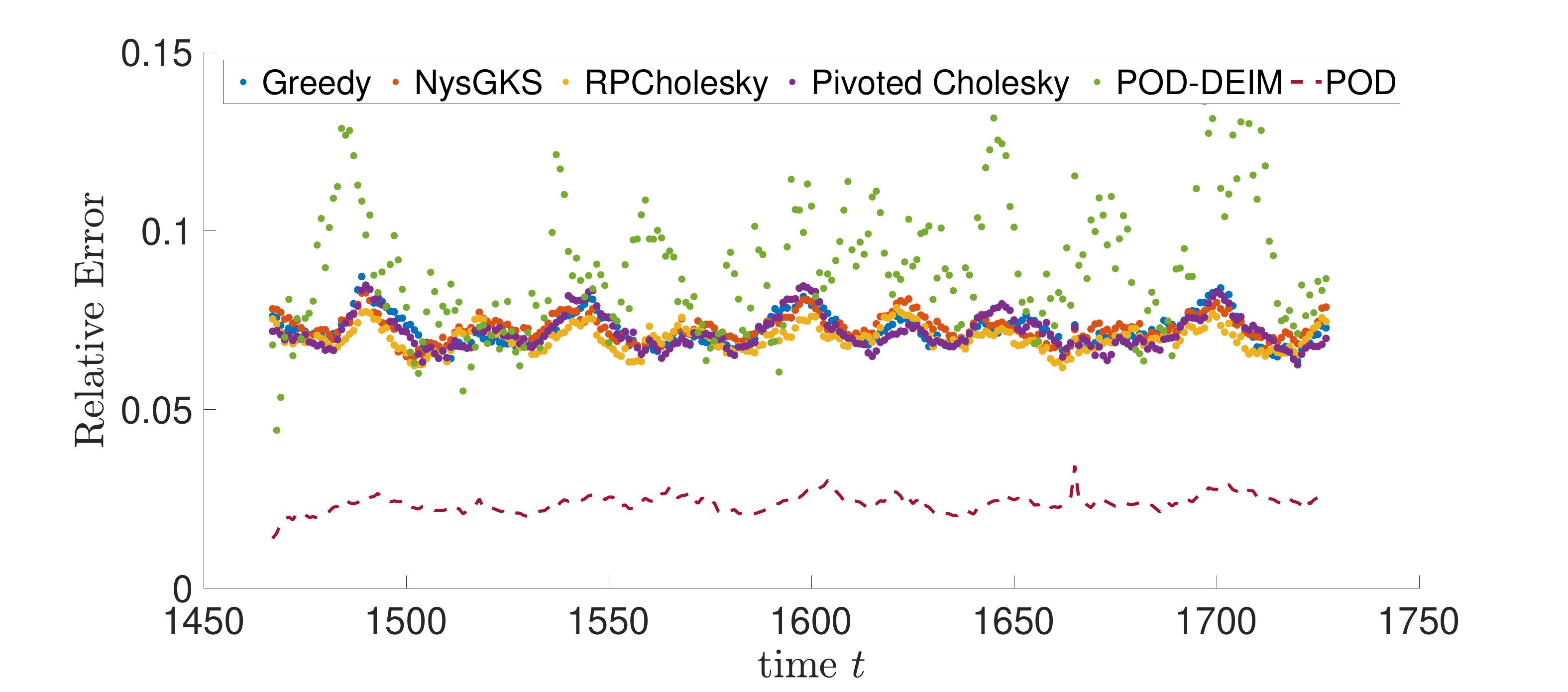}
    \caption{Relative Error in Reconstruction as $t$ increases (including POD).}
    \label{fig:err2d}
\end{figure} 

\paragraph{Unnormalized Relative Reconstruction Error}
We observe the ability for the reconstructions from each method to capture anomalies by removing the mean over the POD training set from the relative error calculation. That is, for each time step, the unnormalized relative error is computed using the formula
\begin{align} \label{eq:unnormrelerr}
    \text{error}_\text{unnormalized} = \frac{\|\vec m(\vec f_p)-\vec f\|_2}{\|\vec f-\mu\|_2},
\end{align}
where $\vec m(\vec f_p)$ is defined in \eqref{eq:meanrecon}, $\vec{f}$ is the vector containing the true function values, and $\mu$ is the mean over the POD training set. We observe that our proposed methods behave similarly to the POD-DEIM method in the presence of anomalies with the exception of a few spikes in error from the POD-DEIM method in \Cref{fig:anomaly}.

\begin{figure}[!ht]
    \centering
    \includegraphics[width=0.75\linewidth]{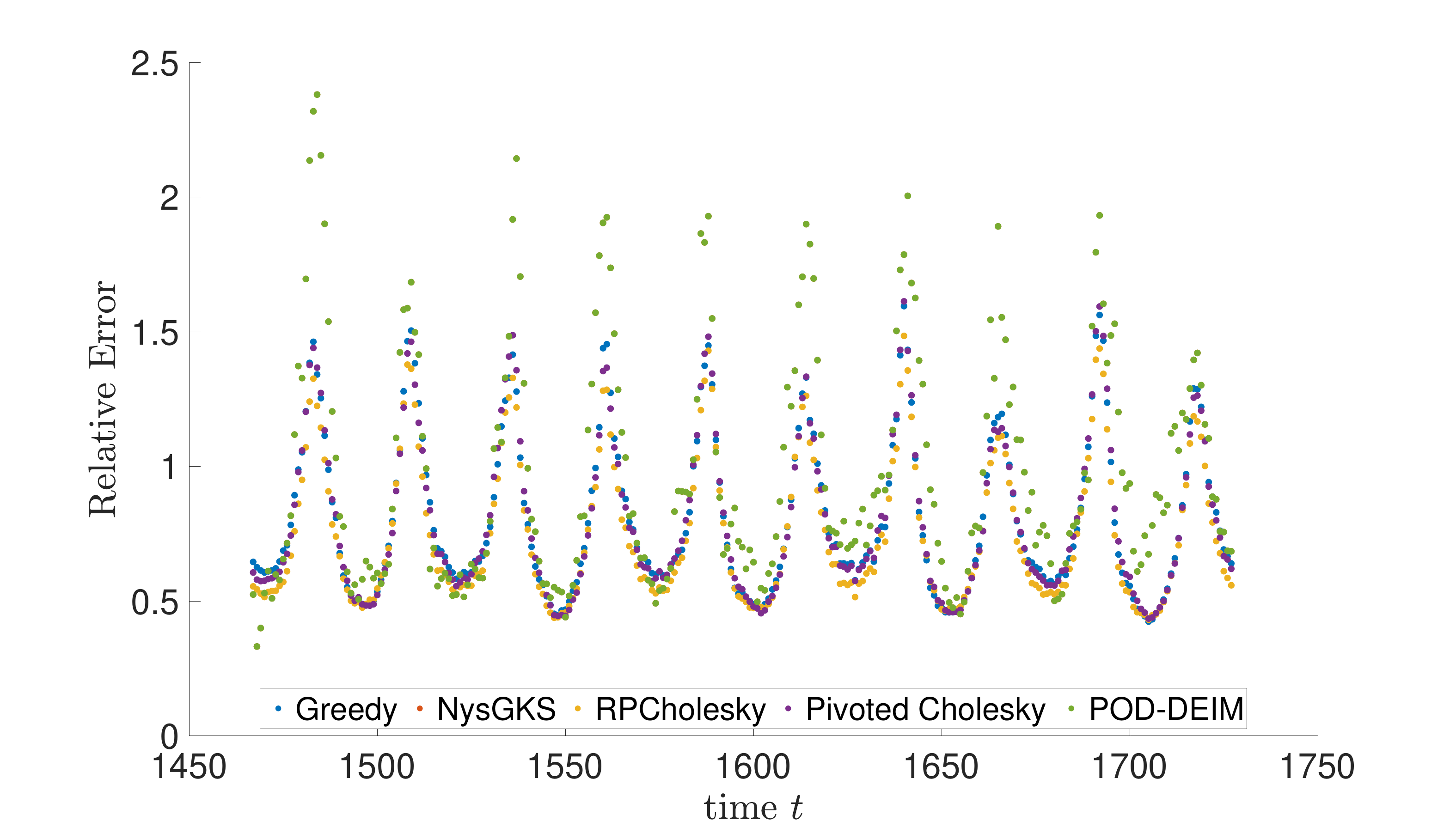}
    \caption{\textbf{Unnormalized} relative error in reconstruction as $t$ increases.}
    \label{fig:anomaly}
\end{figure}

\subsubsection{Placement using great circle Distance}\label{ssec:geodist}
In this section, we take into account the shape of the Earth and consider using the great circle distance between two candidate sensor locations rather than the simple Euclidean distance in the kernel function \eqref{gkernel}. The great circle distance measures the shortest path between two points on a sphere. The Haversine formula computes the great circle between two points when given their latitude and longitude. For our implementation, we use \texttt{MATLAB}'s \texttt{distance()} function with the World Geodetic System of 1984 (WGS84) as the reference ellipsoid to find the distance between candidate sensors in kilometers.

\paragraph{Hyperparameter Tuning} Before running numerical experiments, we must tune the hyperparameters of this new kernel function. We perform a similar hyperparameter sweep as before sweep for $\sigma_f$ and $\ell$. The hyperparameter pair which maximized \eqref{lml} is summarized in the following table:
\begin{center}
    \begin{tabular}{c|c|c}
        $\sigma_f$  & $\ell$& $\eta$ \\ \hline
        8 &  1550 & $  0.033$%\\
    \end{tabular}
    
\end{center}

\paragraph{Reconstruction with Great Circle Distance Based Kernel}
We showcase the sensor selections made by the RPCholesky algorithm (\Cref{alg:cholgks} and \cite{chen2025randomly}) using the great circle distance based kernel. We notice in \Cref{fig:2drecongeo} that, the sensor placements are now more sparse towards the north and south poles. However, the relative reconstruction error over time in \Cref{fig:anomaly3} behaves similarly to the numerical experiments using the Euclidean distance based squared exponential kernel \eqref{gkernel} in \Cref{fig:anomaly}.

\begin{figure}[!ht]
     \centering
     \begin{subfigure}[b]{0.49\textwidth}
     
         \centering
         \includegraphics[width=\textwidth]{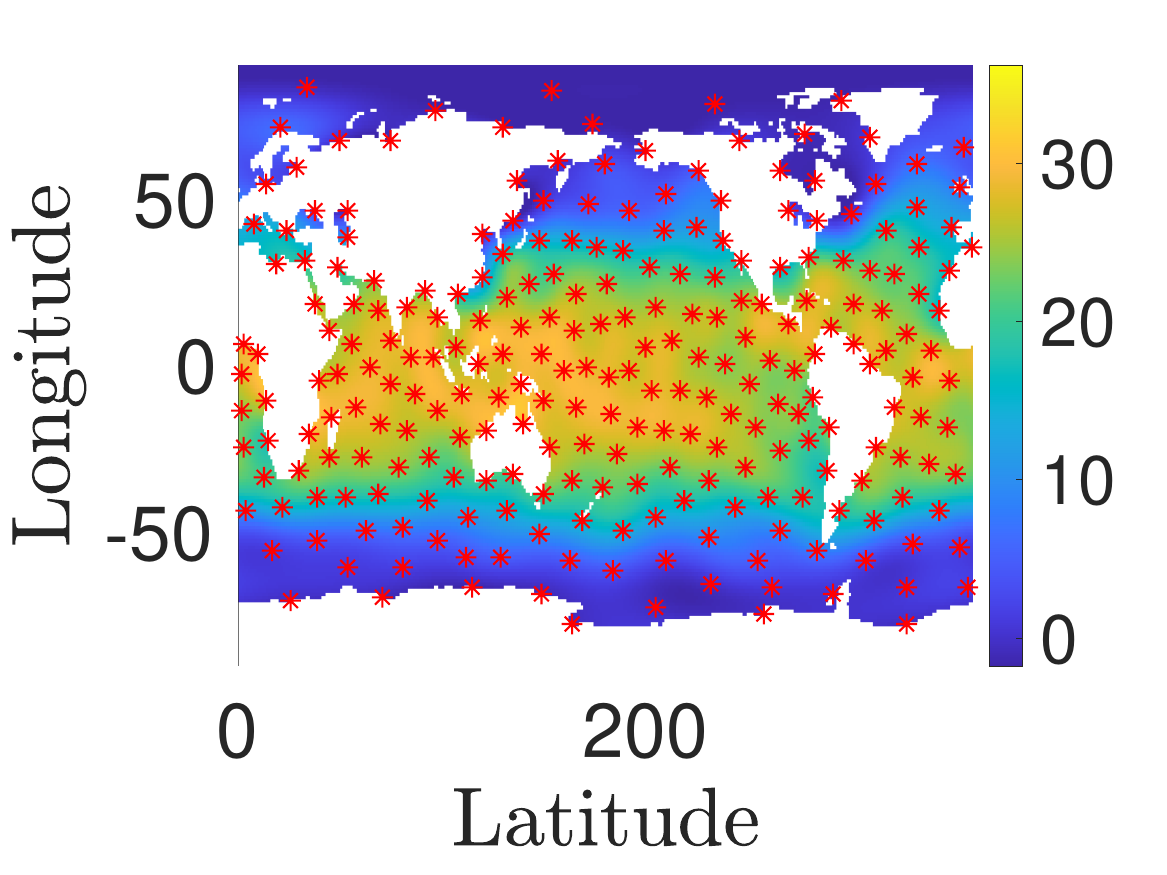}
         
     \end{subfigure}
     \hfill
     \begin{subfigure}[b]{0.49\textwidth}
         \centering
         \includegraphics[width=\textwidth]{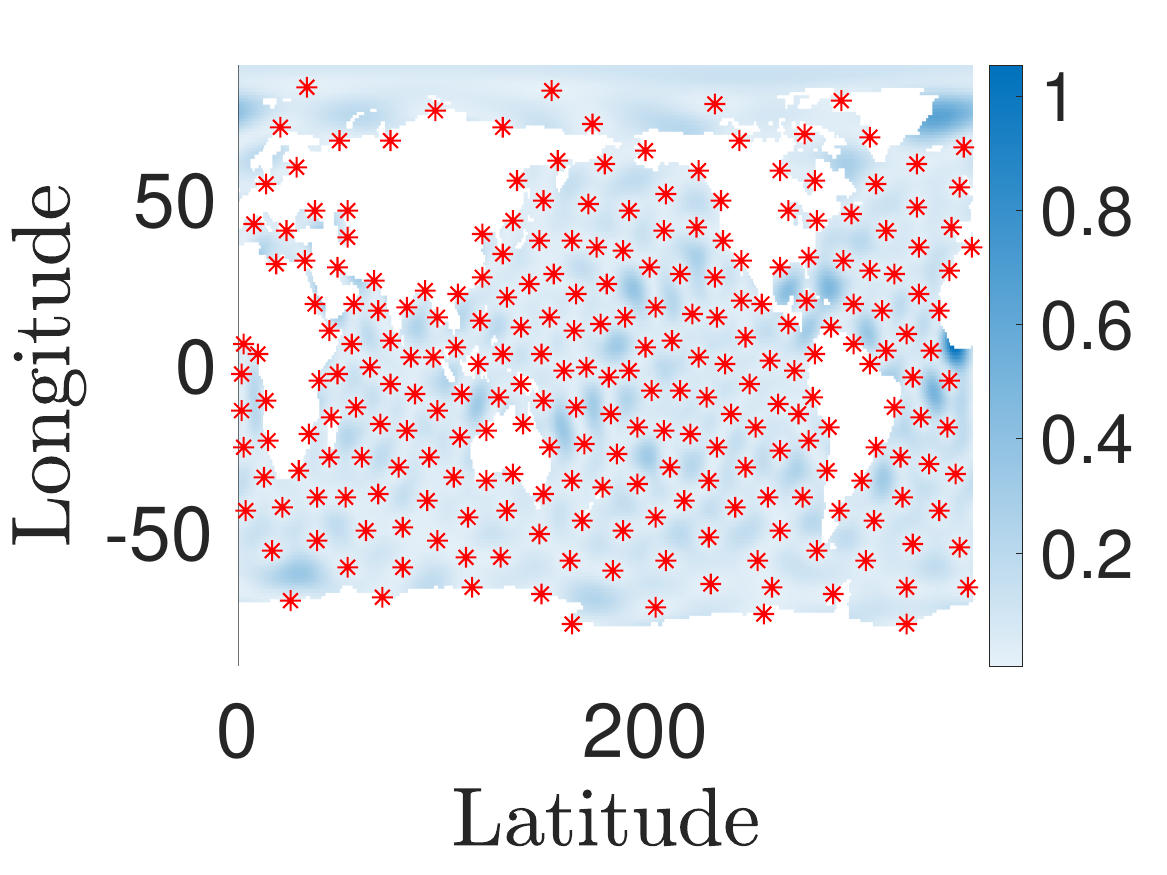}
     \end{subfigure}
        \caption{Random pivoted Cholesky sensor placements (red stars), GP reconstruction, and uncertainty using great circle distance.}
        \label{fig:2drecongeo}
\end{figure}

\begin{figure}[!ht]
    \centering
    \includegraphics[width=0.75\linewidth]{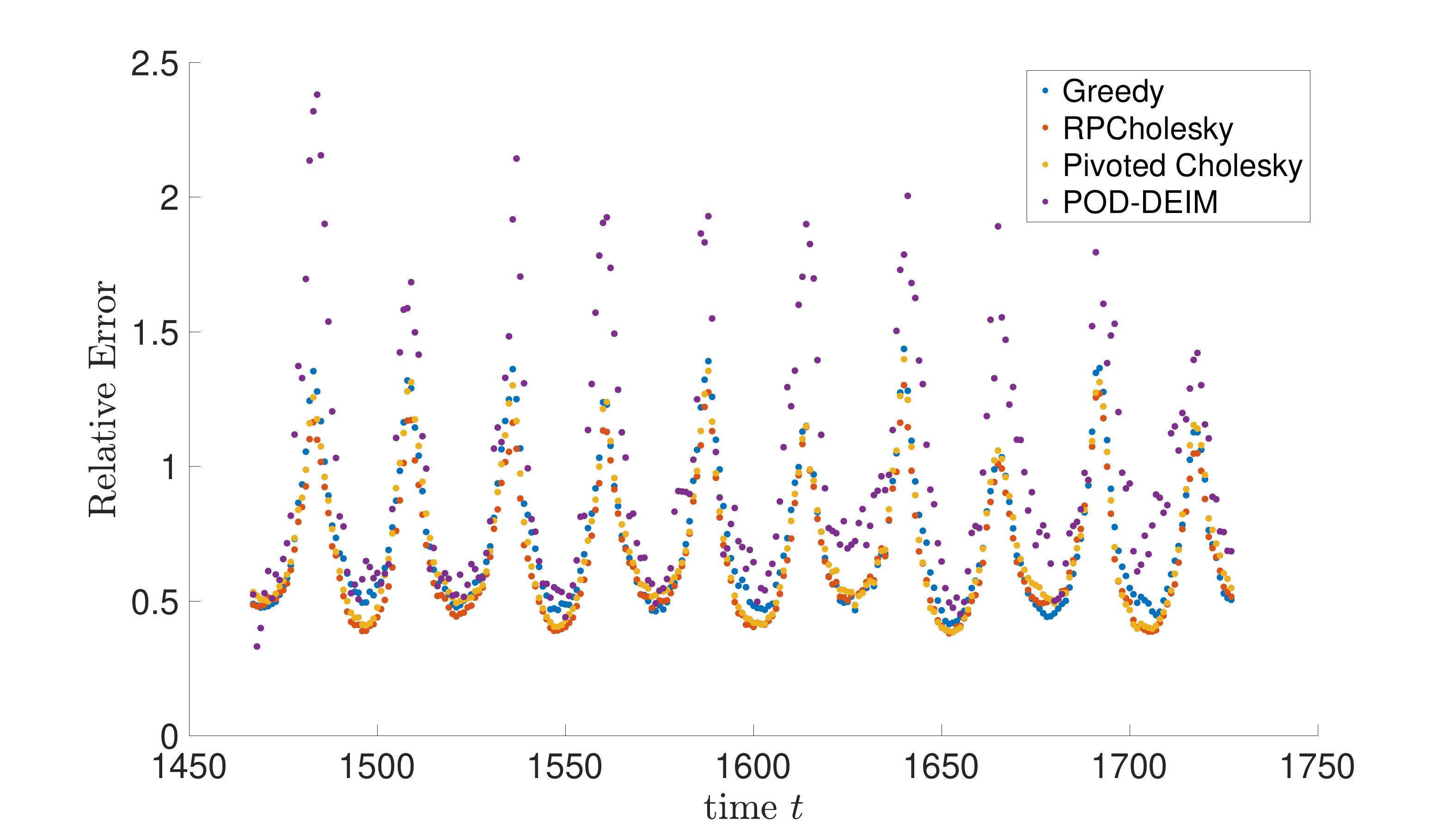}

    \caption{\textbf{Unnormalized} relative error in reconstruction as $t$ increases using great circle distance with POD error.}
    \label{fig:anomaly3}
\end{figure}

\subsection{Four Dimensional Example} \label{ssec:4d}
The following section is a brief demonstration of the use of our algorithm for the purpose of surrogate modeling of complicated behavior in multivariate functions. The Zhou function (1998) \cite{zhou,surrogatedata} is an example of a $D$-dimensional function which is not easily approximated by low-order polynomials. The function is defined as 
\begin{align} \label{eq:4dfunction}
    f(\vec x) &= \frac{10^D}{2} [\varphi(10(\vec x - \zeta_1 \vec e ))+\varphi(10(\vec x - \zeta_2 \vec e))], \quad \text{where} \\
    \varphi(\vec x^*) &= (2 \pi) ^{D/2}\exp(-0.5\|\vec x^*\|_2^2),
\end{align} and $\vec e \in \reals^D$ is the $D$-dimensional vector of all ones. For the ground truth data, we let $\zeta_1 = 1/3$ and $\zeta_2 = 2/3$ in \eqref{eq:4dfunction}. To tune hyperparameters, we set $\zeta_1 = 1/4$ and $\zeta_2 = 3/4$. The value of hyperparameters are chosen to be: 
\begin{center}
    \begin{tabular}{c|c|c}
    %\hline
        $\sigma_f$  & $\ell$& $\eta$ \\ \hline
        6.5 &  0.16 & $  0.2845$%\\
     %   \hline % was \sigma_f = 16 and \ell = 0.11
    \end{tabular}
\end{center}

\begin{figure}[!ht]
    \centering
    \includegraphics[width=0.75\linewidth]{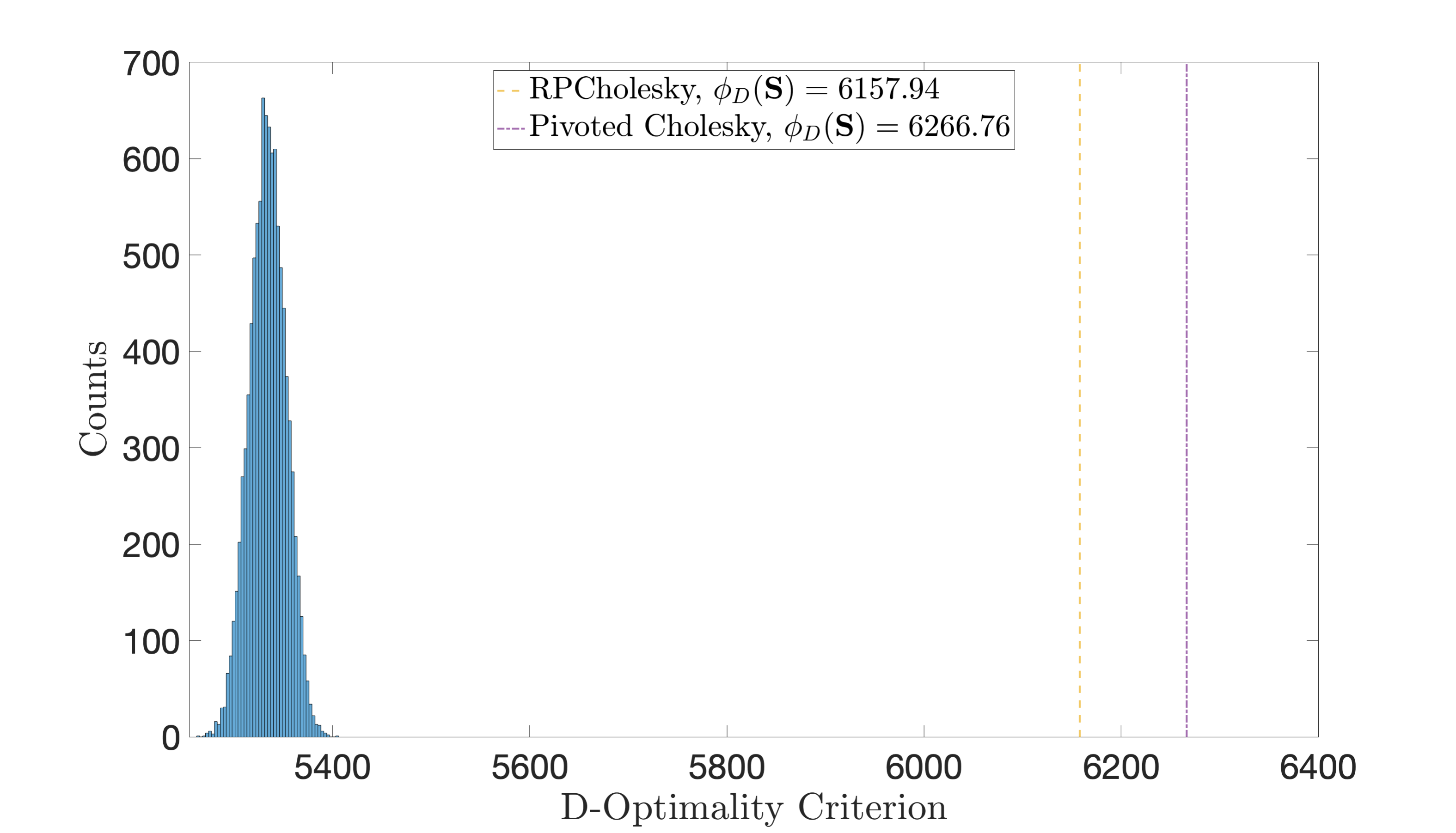}
    \caption{RPCholesky and greedy pivoted Cholesky (\Cref{alg:cholgks}) compared with 10,000 realizations fo random sensor placement in terms of D-optimality. The corresponding relative errors are $8.84\%$ and $10.12\%$, for the RPCholesky and greedy pivoted Cholesky selections, respectively.}
    \label{fig:4D_example}
\end{figure}

In this experiment, we use Latin Hypercube sampling to construct a set of $n=10,000$ candidate sensors and set $k = 3,000$. We employ the RPCholesky and greedy pivoted Cholesky algorithms (\Cref{alg:cholgks}) to select a D-optimal designs. We provide in \Cref{fig:4D_example} a D-optimality comparison of 10,000 random realizations of sensor selection to the results of our algorithms. Again, we see that the D-optimality of the proposed methods is significantly larger than random selections. The D-optimality of the selection made by the RPCholesky + GKS and greedy pivoted Cholesky + GKS algorithm is 6157.94 and 6266.76, respectively. The relative error in the reconstruction is $8.84\%$ and $10.12\%$ respectively.

\section{Additional Proofs}\label{sec:proofs}
In this section, we present the proof to \Cref{thm2}.
\subsection{Bounds on Nystr\"om Approximation Eigenvalues}\label{ssec:nysbound}
\begin{proof}[Proof of \Cref{thm2}]\label{nyseigproof}

    The proof technique is similar to \cite[Theorem 9]{saibaba2019randomized}. We prove the sequence of inequalities from left to right. For the first inequality, by~\cite[Lemma 2.1, Item 3]{frangella2023randomized} $\mat{K} \succeq \wh{\mat{K}}$.  From the properties of Loewner partial ordering~\eqref{eqn:weylconsq}, follows $  \lambda_i(\mat K) \geq\lambda_i(\wh{\mat K})$ for $1 \le i \le k$, establishing the first set of inequalities.

    We now prove the second set of inequalities. We may rewrite the matrix $\mat Z := \mat K^{1/2} \mat \Omega$ as the following 
    \begin{align*}
        \mat Z := \mat K^{1/2} \mat \Omega = \mat V \mat \Lambda^{1/2} \begin{bmatrix}
            \mat V_k^\top \mat \Omega \\ \mat V_\perp^\top \mat \Omega 
        \end{bmatrix} = \mat V \begin{bmatrix}
            \mat \Lambda_k^{1/2} \mat \Omega_k \\ \mat \Lambda_\perp^{1/2} \mat \Omega_{n-k}
        \end{bmatrix} .
    \end{align*}
    Note that $$\begin{aligned}\widehat{\mat{K}} =  \mat{K}^{1/2} (\mat{Z}(\mat{Z}\t\mat{Z})^\dagger \mat{Z}\t)\mat{K}^{1/2}  =  \mat{K}^{1/2}\mat{P}_{\mat{Z}}\mat{K}^{1/2}. \end{aligned}$$
    Define the matrix $\wh{\mat Z}$  
    \begin{align}
        \widehat{\mat Z} := \zhatdef = \mat V \begin{bmatrix}
\mat I_k \\ \mat F        
    \end{bmatrix} \in \reals^{n \times k}
    \end{align}
    where $\mat F := \mat \Lambda_\perp^{1/2} \mat \Omega_{n-k}\mat \Omega_k^\dagger \mat \Lambda_k^{-1/2} \in \reals^{(n-k) \times k}$ and we have used the fact that  $\mat \Omega_k$ has a right pseudo-inverse. We define a new orthogonal projector $\mat{P}_{\widehat{\mat{Z}}}$ onto the range of $\widehat{\mat{Z}}$.  Note that $\range(\widehat{\mat Z}) \subseteq \range(\mat Z)$, so that by \cite[Propositon 8.5]{HMT}, we have $
         \mat P_{\wh{\mat Z}} \preceq \mat P_\mat Z$ and 
         \[ \widehat{\mat{K}} =  \mat{K}^{1/2}\mat{P}_{\mat{Z}}\mat{K}^{1/2} \succeq \mat K^{1/2} \mat P_{\zhat} \mat K^{1/2}.  \] 
    With this majorization, we now look at 
    \begin{align*}
        \mat K^{1/2} \mat P_{\zhat} \mat K^{1/2} & = \mat{V\Lambda}^{1/2} \begin{bmatrix}
\mat I_k \\ \mat F        
    \end{bmatrix} (\mat I_k + \mat F^\top \mat F)^{-1} \begin{bmatrix}
        \mat I_k & \mat F^\top
    \end{bmatrix} \mat{\Lambda}^{1/2} \mat V \\
    & = \mat V \begin{bmatrix}
        \mat \Lambda_k^{1/2} (\mat I_k + \mat F^\top \mat F)^{-1} \mat \Lambda_k^{1/2} & * \\ * & *
    \end{bmatrix} \mat V^\top,
    \end{align*}
    where $*$ denotes sub-blocks that are unimportant to the present calculation.
    Since eigenvalues are preserved  under orthogonal similarity transformations, we may look at the eigenvalues of $\mat{H} : = \mat \Lambda_k^{1/2} (\mat I_k + \mat F^\top \mat F)^{-1} \mat \Lambda_k^{1/2}$. By the properties of Loewner partial ordering and Cauchy's interlacing theorem \eqref{cauchyinterlacing}, 
    \begin{align}\label{eqn:nystrominter}
      \lambda_i(\widehat{\mat{K}}) = \lambda_i( \mat K^{1/2} \mat P_{\mat{Z}} \mat K^{1/2})\geq  \lambda_i(\mat K^{1/2} \mat P_{\zhat} \mat K^{1/2}) \geq  \lambda_i(\mat{H} ), \qquad 1 \le i \le k.
    \end{align}
    Thus, we focus on the eigenvalues of the matrix $\mat{H}$.
    
    Applying \Cref{matbounds} with $\mat A = \mat \Lambda_\perp^{1/2} \mat  \Omega_{n-k} \mat \Omega_k^\dagger$ and  $\mat B = \mat \Lambda_k^{-1/2}$, it follows that
    \begin{align*}
        \mat F\t \mat F \preceq \| \mat \Lambda_\perp^{1/2} \mat  \Omega_{n-k} \mat \Omega_k^\dagger\|_2^2 \mat \Lambda_k^{-1}.
    \end{align*}
   By~\eqref{eqn:loewnerinter}, we have
    $\mat{H}\succeq\mat \Lambda_k^{1/2}(\mat I_k  + \|\mat \Lambda_\perp^{1/2} \mat  \Omega_{n-k} \mat \Omega_k^\dagger\|_2^2 \mat \Lambda_k^{-1})^{-1}\mat \Lambda_k^{1/2}$. Then, the eigenvalues of $ \mat{H}$ are bounded from below as
    \begin{align*}
    \lambda_i(\mat{H}) \geq  \frac{\lambda_i(\mat K)}{1 +  \|\mat \Lambda_\perp^{1/2} \mat \Omega_{n-k}\mat \Omega_k^\dagger\|_2^2 \lambda_i(\mat \Lambda_k^{-1})}, \qquad 1 \le i \le k,
    \end{align*}
    since $\mat \Lambda_k$ is a diagonal matrix containing the top $k$ eigenvalues of $\mat{K}$. Plug this inequality into~\eqref{eqn:nystrominter} to complete the proof. 
\end{proof}

\end{document}